\documentclass[11pt]{article}

\usepackage[table]{xcolor}
\usepackage{amsfonts}
\usepackage{amsmath}
\usepackage{amsthm}
\usepackage{amssymb}
\usepackage{mathrsfs}
\usepackage{mathdots}
\usepackage{graphicx}
\usepackage{algorithmic}
\usepackage{verbatim}
\usepackage{pgfplotstable}
\usepackage{paralist}
\usepackage{bm}
\usepackage{url}

\def\scrH{\mathscr{H}}

\def\scrS{\mathscr{S}}

\def\scrV{\mathscr{V}}

\def\calA{\mathcal{A}}
\def\calB{\mathcal{B}}

\def\calD{\mathcal{D}}

\def\calI{\mathcal{I}}

\def\calK{\mathcal{K}}

\def\calR{\mathcal{R}}

\def\bfS{\mathbf{S}}

\def\bbC{\mathbb{C}}

\def\bbN{\mathbb{N}}

\def\bbR{\mathbb{R}}

\newcommand{\mvec}{\mathop{\text{vec}}}
\newcommand{\rrank}{\mathop{\text{rank}}}

\newcommand{\colspan}{\mathop{\text{colspan}}}

\newcommand{\minim}{\mathop{\text{\;minimize\;}}}

\newcommand{\bmx}{\begin{bmatrix}}
\newcommand{\emx}{\end{bmatrix}}
\newcommand{\bsm}{\left[\begin{smallmatrix}}
\newcommand{\esm}{\end{smallmatrix}\right]}

\newcommand{\ie}{\emph{i.e.}}
\newcommand{\eg}{\emph{e.g.}}

\def\bbZp{\mathbb{N}}

\newcommand{\set}[1]{\mathcal{#1}} 

\newcommand{\arr}[1]{\mathbf{#1}}

\newcommand{\matr}[1]{\boldsymbol{#1}}

\newcommand{\vect}[1]{\boldsymbol{#1}}

\makeatletter
\newcommand{\T}{{\sf T}}      
\renewcommand{\H}{{\sf H}}      
\newcommand{\rank}[1]{\mathop{\operator@font rank}\{#1\}}   
\newcommand{\krank}[1]{\mathop{\operator@font krank}\{#1\}}
\newcommand{\trace}[1]{\mathop{\operator@font trace}\{#1\}}
\newcommand{\Diag}[1]{\mathop{\operator@font Diag}\{#1\}}    
\newcommand{\diag}[1]{\mathop{\operator@font diag}\{#1\}}    
\newcommand{\Span}[1]{\mathop{\operator@font Span}\{#1\}}    
\renewcommand{\vec}{\mathop{\operator@font vec}} 
\makeatother

\newcommand{\eqdef}{\stackrel{\mathrm{def}}{=}}

\def\trgset#1#2{\blacktriangle^{(#1,#2)}}
\def\degset#1#2{\triangle^{(#1,#2)}}
\def\nvar{m}
\def\lker#1{\mathop{\mathrm{l.ker.}}{(#1)}}
\def\iarr#1{\arr{#1}}
\def\struct#1{\scrS(#1)}
\def\hmatr#1#2{\scrH_{#1} {(#2)}}
\def\vmatr#1#2{\scrV_{#1} {(#2)}}

\def\hrank#1{\mathop{\mathrm{hrank}}{(#1)}}
\def\inter#1{\mathop{\mathrm{int}}{(#1)}}

\newenvironment{smallarray}[1]
 {\null\,\vcenter\bgroup\scriptsize
  \arraycolsep=.13885em
  \hbox\bgroup$\array{@{}#1@{}}}
 {\endarray$\egroup\egroup\,\null}
 
 \newenvironment{mediumarray}[1]
 {\null\,\vcenter\bgroup\small
  \arraycolsep=.13885em
  \hbox\bgroup$\array{@{}#1@{}}}
 {\endarray$\egroup\egroup\,\null}
 
\newenvironment{enumeratex}{\begin{inparaenum}[\begin{bfseries}(i)\end{bfseries}]}{\end{inparaenum}}

\def\smallhline{\noalign{\vskip1pt\hrule\vskip1pt}}

\newtheorem{theorem}{Theorem}
\newtheorem{problem}{Problem} 
 
\newtheorem{corollary}[theorem]{Corollary} 
\newtheorem{lemma}[theorem]{Lemma} 
\newtheorem{proposition}[theorem]{Proposition} 
\newtheorem{definition}[theorem]{Definition} 
\newtheorem{note}[theorem]{Remark} 
\newtheorem{remark}[theorem]{Remark} 
\newtheorem{example}{Example}

\def\proofapp{\begin{proof}See Appendix \ref{proofcomplex-sec} for a proof.\end{proof}}
\begin{document}

\author{Konstantin Usevich and Pierre Comon \\
{\small CNRS, GIPSA-Lab, F-38000 Grenoble, France.} \\ {\small E-mail: \tt{firstname.lastname@gipsa-lab.fr}}}

\title{Quasi-Hankel low-rank matrix completion:\\ a convex relaxation}
\maketitle

\begin{abstract}
The completion of matrices with missing values under the rank constraint is a non-convex optimization problem. A popular convex relaxation is based on minimization of the nuclear norm  (sum of singular values) of the matrix. For this relaxation, an important question is whether the two optimization problems lead to the same solution. This question was addressed in the literature mostly in the case of random positions of missing elements and random known elements. In this contribution, we analyze the case of structured matrices with fixed pattern of missing values, in particular, the case of Hankel and quasi-Hankel matrix completion, which appears as a subproblem in the computation of symmetric tensor canonical polyadic decomposition. We extend existing results on completion of rank-one real Hankel matrices to completion of rank-r complex Hankel and quasi-Hankel matrices.
\end{abstract}
\section{Introduction}
The problem of completing matrices with missing entries can be traced back to the works of Prony in 1795, and has been addressed since in various fields including: 
compressed sensing \cite{Chen.Chi14IToIT-Robust,CandR09:fcm,Gros11:ieeeit},
system identification and control \cite{DeSc00:jcam,FazePST13:simax,LiuV09:simax}, 
graph theory \cite{Tros00:coa}, 
collaborative filtering \cite{CandP10:pieee}, 
information theory \cite{GrayD05}, 
chemometrics \cite{Bro97:cils}, 
seismics \cite{Hogb03:laa,KreiSS13:geophysics},   
estimation problems and sensor networks \cite{CandP10:pieee}, to cite a few.
It also appears as a subproblem in the computation of symmetric tensor Canonical Polyadic (CP) decompositions \cite{Brachat.etal10LAaiA-Symmetric}.

\subsection{Structured matrix completion}
We are interested in structured matrices of the form
\begin{equation}\label{eq:affine_structure}
\struct{\vect{p}} = \matr{S}_0 + \sum\limits_{k=1}^{N} p_k \matr{S}_k,
\end{equation}
where $\matr{S}_k$, $k \in \{0,\ldots, N\}$ are known linearly independent $n\times n$ matrices, and $\vect{p}=[p_1,\dots, p_N]^{\T}$.
Without loss of  generality, we assume that the matrices $\matr{S}_k \in \bbC^{n\times n}$ are symmetric (a nonsymmetric matrix completion problem can be always symmetrized).
Thus, $\mathscr{S}$ is an injective map $\mathbb{C}^N \to \mathbb{C}^{n\times n}$, called affine matrix structure.
The Structured Low-Rank Matrix Completion (SLRMC) for complex affine matrix structures is then stated  formally as
\begin{equation}\label{eq:rankmin}
\widetilde{\vect{p}} =  \arg \min_{\vect{p} \in \bbC^{N}} \rank{\struct{\vect{p}}}.
\end{equation}
In the literature on matrix completion, the problem~\eqref{eq:rankmin} is often formulated in dual form, \textit{i.e.}, as rank minimization subject to linear constraints, and is called affine rank minimization \cite{Recht.etal10SR-Guaranteed}.
In this paper we prefer the formulation \eqref{eq:rankmin}, which seems to be more suitable for completion of structured (\textit{e.g.}, Hankel-like) matrices.

There are special cases when the solution to SLRMC can be found in polynomial time: partially known block-triangular matrices \cite{Woerdeman89LAaiA-Minimal}, Hankel/Toeplitz matrices \cite{Iohvidov82-Toeplitz,Heinig.Rost84-Algebraic}, block-Hankel matrices \cite{Feldmann.Heinig99IEaOT-Parametrization} and some cases of quasi-Hankel matrices \cite{Laurent.Mourrain09AdM-generalized}.
In the general case, the low-rank matrix completion is NP-hard (see, \eg, \cite{Fazel02PhD-Matrix}).

SLRMC can be also considered as an extreme case of Structured Low-rank matrix Approximation (SLRA) problem with missing data \cite{Markovsky.Usevich13SJMAA-Structured}.
The latter problem is more relevant in practice, where we may have noise in addition to missing data.

\subsection{A convex relaxation}
A  popular approach in machine learning \cite{Fazel02PhD-Matrix,Recht.etal10SR-Guaranteed} is to build a  relaxation of LRMC, by replacing the rank with the nuclear norm (\textit{i.e.} the sum of singular values): 
\begin{equation}\label{eq:nnmin}
\widehat{\vect{p}} = \arg \min_{\vect{p} \in \bbC^{N}} \| \scrS(\vect{p}) \|_*.
\end{equation}
Now \eqref{eq:nnmin} is a convex optimization problem, and a variety of convex optimization methods can be used to find its global minimum.
A central question  is when do the solutions of \eqref{eq:nnmin} and \eqref{eq:rankmin} coincide (\ie, the conditions on low-rank matrix recovery with the nuclear norm).

A special case of SLRMC is when each matrix $\matr{S}_k$ contains only one nonzero element.
This is the most common unstructured matrix completion: $\matr{S}_0$ is the known part of the matrix and the positions of nonzero elements in $\matr{S}_k$ indicate the values to be completed.
There is a vast literature on the subject on low-rank matrix recovery, where mainly 
the case of unstructured matrix completion is treated.
Most of the results are proved in a random setting  \cite{CandR09:fcm,Recht.etal10SR-Guaranteed,Gros11:ieeeit}: the positions of missing entries (non-zero elements of $\matr{S}_k$, $k \in \{1,\ldots, N\}$) are assumed to be drawn randomly; often the known entries  are also assumed to be random.
There are also results on the general affine rank minimization (which is equivalent to SLRMC), but the linear constraints are also assumed to be drawn randomly according to a certain probability distribution.

Up to the authors' knowledge, there exists only one result \cite{Dai.Pelckmans15A-nuclear}, which 
treats the case of fixed structure, in a very simple case.
\begin{theorem}[{\cite[Thm. 1]{Dai.Pelckmans15A-nuclear} }]\label{thm:rank_one}
Let $\mathscr{S}$ be  a Hankel matrix
\[
\mathscr{S}(\vect{p}) = 
\begin{bmatrix}
1 & \lambda &  \cdots & \lambda^{n} \\
\lambda & \lambda^2 &  \iddots & p_1 \\
\vdots &    \iddots &  \iddots & \vdots \\
\lambda^{n} & p_1 &  \cdots & p_{n-1} \\
\end{bmatrix},
\]
where $\lambda \in (-1;1)$. Then the solution of \eqref{eq:nnmin}, with  $\vect{p} \in \mathbb{R}^{n-1}$, is unique and coincides with the minimal rank (rank-$1$) completion, which is given by
$p_k  = \lambda^{n+k}$.
\end{theorem}
In this paper we extend the results of \cite{Dai.Pelckmans15A-nuclear} in two directions: (i) to  Hankel matrices with arbitrary rank, and (ii) to quasi-Hankel matrices, which are particularly interesting in the context of symmetric tensor CP decomposition \cite{Brachat.etal10LAaiA-Symmetric}.
As in \cite{Dai.Pelckmans15A-nuclear}, we consider cases when the solution to SLRMC is known, and establish conditions on $\mathscr{S}$ when the solution of \eqref{eq:rankmin} and \eqref{eq:nnmin} coincide.

Our results are more general (for example, applicable both in the real and complex cases), and do not depend on the results of \cite{Dai.Pelckmans15A-nuclear}.
However, we believe that our exposition may also help to understand the complicated logic of the proof in \cite{Dai.Pelckmans15A-nuclear}.

In Section~\ref{sec:background} we introduce the main notation, in particular quasi-Hankel matrices, which are generalizations of Hankel matrices.
In Section~\ref{seq:qh_completion} we formally state the considered quasi-Hankel matrix completion problem, and summarize some known results
(this summary is needed for stating the main results of the paper in Section~\ref{sec:results}). 
In Section~\ref{sec:nnmin_genprop} we recall optimality conditions for convex optimization problems, and specialize them to nuclear norm minimization for structured matrices.

Section~\ref{sec:lemmas} contains the main lemmas needed for proofs of the results of the paper. The lemmas imply that the optimality conditions are satisfied if the column space of completed matrices is not too far from a certain simply structured subspace.
Finally, in Section~\ref{sec:results}, we prove the main results of the paper, and illustrate them with numerical experiments in Section~\ref{sec:numerical}.

\section{Background and main notation}\label{sec:background}
\subsection{Sets of multi-indices}
We denote by $\mathbb{N}$ the set of nonnegative integers.
First, for a multi-index $\vect{\alpha} = ({\alpha_1}, {\alpha_2}, \cdots, {\alpha_m}) \in \mathbb{N}^m$, the monomial  $x_1^{\alpha_1} x_2^{\alpha_2} \cdots x_m^{\alpha_m}$  will be denoted as
$\vect{x}^{\vect{\alpha}}$, 
and its total degree is $|\vect{\alpha}|\eqdef\sum_\ell \alpha_\ell$.

For ordering the multi-indices, we use in this paper an ordering denoted by $\prec$. It is defined recusively on the size of vectors $\vect\alpha$ and $\vect\beta$ as:
\[
\begin{split}
& \vect\alpha \prec \vect\beta \Leftrightarrow\\
& \Leftrightarrow \left\{\begin{array}{l}
|\vect\alpha| < |\vect\beta| \\
 \mbox{ or } \\
|\vect\alpha| = |\vect\beta| ~ \mbox{and} ~ (\alpha_2, \ldots, \alpha_{\nvar}) 
\prec  (\beta_2, \ldots, \beta_{\nvar})
\end{array}\right.
\end{split}
\]

Next, we shall denote by $\trgset{\nvar}{d}$ the set of multi-indices $\{\vect{\alpha}\in \mathbb{N}^{\nvar}\!:\, |\vect{\alpha}| \le d\}$.
By $\degset{\nvar}{d}$ we denote the set $\{\vect\alpha \in \bbZp^{\nvar}:\, |\vect\alpha| = d\}$. 
It is easy to see that 
\[
\trgset{\nvar}{d} = \degset{\nvar}{0} \cup\degset{\nvar}{1} \cup \cdots \cup \degset{\nvar}{d}.
\]
\begin{example}
Take $m=2$ and $d=3$. Then
\[
\begin{split}
\trgset{2}{3} =& \{{(0,0)}, {(1,0)}, {(0,1)}, {(2,0)}, {(1,1)}, {(0,2)},\\ 
               & \phantom{\{}{(3,0)}, {(2,1)}, {(1,2)}, {(0,3)}\},\\
\degset{2}{0} =& \{ {(0,0)} \}, ~ \degset{2}{1} = \{{(1,0)}, {(0,1)}\},\\
\degset{2}{2} =& \{ {(2,0)}, {(1,1)}, {(0,2)}\}, \\
\degset{2}{3} =& \{ {(3,0)}, {(2,1)}, {(1,2)}, {(0,3)} \}.
\end{split}
\]
The multi-indices  are ordered according to $\prec$ in the sets defined above. In fact, when the total degrees are the same, terms are sorted only by increasing degrees in the second variable.
\end{example}
 
For sets $\set{A}, \set{B} \subset \mathbb{N}^{\nvar}$, we define their Minkowski sum  as
$
\set{A} + \set{B} \eqdef \{ \vect{\alpha} + \vect{\beta} \,:\, \vect{\alpha}\in\!\set{A}, \vect{\beta}\in\!\set{B} \},
$
with a shorthand notation $2\set{A} \eqdef \set{A} + \set{A}$.
It is easy to see that $\degset{\nvar}{d_1} + \degset{\nvar}{d_2} = \degset{\nvar}{d_1+d_2}$ and $\trgset{\nvar}{d_1} + \trgset{\nvar}{d_2} = \trgset{\nvar}{d_1+d_2}$.

For $m=1$, we have that $\trgset{1}{d} = \{0,\ldots,d\}$ and 
$\{0,\ldots,d_1\}+ \{0,\ldots,d_2\} = \{0,\ldots,d_1+d_2\}$.
For $m = 2$, an example is shown in Fig.~\ref{fig:msum} (the multi-indices are depicted as black dots).
\begin{figure}[ht!]
\centering
\begin{tikzpicture}[scale=0.6]
\tikzstyle{dot} = [draw,circle,minimum size=1.25mm,fill,inner sep=0mm]
\begin{scope}[xshift = 0cm]
    \draw (0,1.5) node[anchor=east]  {\small$3$};
    \draw (0,2.05) node[anchor=east]  {\small$\alpha_2$};
    \draw (0.5,0) node[anchor=north]  {\small$1$};
    \draw (1,0) node[anchor=north ]  {\small$2$};
    \draw (0,0) node[anchor=north ]    {\small$0$};
    \draw (0,0.5) node[anchor= east]    {\small$1$};
    \draw (0,1) node[anchor= east]    {\small$2$};
    \draw (1.5,0) node[anchor=north ]  {\small$3$};
    \draw (2,0) node[anchor=north ]  {\small$\alpha_1$};
 
    \draw[color=black, help lines, line width=.1pt] (0,0)
      grid[xstep=0.5cm, ystep=0.5cm] (1.25,1.25);
 
    \draw(0,0)      node[dot]   (p00) {};
    \draw(0,0.5)      node[dot]   (p01) {};
    \draw(0,1)      node[dot]   (p01) {};
    \draw(0.5,0)      node[dot]   (p10) {};
    \draw(0.5,0.5)      node[dot]   (p11) {};
    \draw(1,0)      node[dot]   (p20) {};
    \draw(1.5,0)      node[dot]   (p01) {};
    \draw(1,0.5)      node[dot]   (p10) {};
    \draw(0.5,1)      node[dot]   (p11) {};
    \draw(0,1.5)      node[dot]   (p20) {};
 
    \draw [->] (0,0) -- (0,2);
    \draw [->] (0,0) -- (2,0);
    
    \draw(1.85,0.6) node {$+$};
    \draw(4.45,0.6) node {$=$};
\end{scope}   
  
\begin{scope}[xshift = 2.9cm]
    \draw (0,1.5) node[anchor=east]  {\small$3$};
    \draw (0,2.05) node[anchor=east]  {\small$\alpha_2$};
    \draw (0.5,0) node[anchor=north]  {\small$1$};
    \draw (1,0) node[anchor=north ]  {\small$2$};
    \draw (0,0) node[anchor=north ]    {\small$0$};
    \draw (0,0.5) node[anchor= east]    {\small$1$};
    \draw (0,1) node[anchor= east]    {\small$2$};
    \draw (1.5,0) node[anchor=north ]  {\small$3$};
    \draw (2,0) node[anchor=north ]  {\small$\alpha_1$};
 
    \draw[color=black, help lines, line width=.1pt] (0,0)
      grid[xstep=0.5cm, ystep=0.5cm] (1.25,1.25);
 
    \draw(0,0)      node[dot]   (p00) {};
    \draw(0,0.5)      node[dot]   (p01) {};
    \draw(0,1)      node[dot]   (p02) {};
    \draw(0.5,0)      node[dot]   (p10) {};
    \draw(0.5,0.5)      node[dot]   (p11) {};
    \draw(1,0)      node[dot]   (p20) {};
     \draw(1.5,0)      node[dot]   (p01) {};
    \draw(1,0.5)      node[dot]   (p10) {};
    \draw(0.5,1)      node[dot]   (p11) {};
    \draw(0,1.5)      node[dot]   (p20) {};

    \draw [->] (0,0) -- (0,2);
    \draw [->] (0,0) -- (2,0);
\end{scope}    

\begin{scope}[xshift = 5.6cm]
    \draw (0,3.5) node[anchor=east]  {\small$\alpha_2$};
    \draw (0.5,0) node[anchor=north]  {\small$1$};
    \draw (1,0) node[anchor=north ]  {\small$2$};
    \draw (1.5,0) node[anchor=north ]  {\small$3$};
    \draw (2,0) node[anchor=north ]  {\small$4$};
    \draw (2.5,0) node[anchor=north ]  {\small$5$};
    \draw (3,0) node[anchor=north ]  {\small$6$};
    \draw (0,0) node[anchor=north ]    {\small$0$};
    \draw (0,0.5) node[anchor= east]    {\small$1$};
    \draw (0,1) node[anchor= east]    {\small$2$};
    \draw (0,1.5) node[anchor= east]    {\small$3$};
    \draw (0,2) node[anchor= east]    {\small$4$};
    \draw (0,2.5) node[anchor= east]    {\small$5$};
    \draw (0,3) node[anchor= east]    {\small$6$};
    \draw (3.5,0) node[anchor=north ]  {\small$\alpha_1$};
 
    \draw[color=black, help lines, line width=.1pt] (0,0)
      grid[xstep=0.5cm, ystep=0.5cm] (3.25,3.25);
 
    \draw(0,0)      node[dot]   (p00) {};
    \draw(0,0.5)      node[dot]   (p01) {};
    \draw(0,1)      node[dot]   (p02) {};
    \draw(0,1.5)      node[dot]   (p03) {};
    \draw(0,2)      node[dot]   (p04) {};
    \draw(0,2.5)      node[dot]   (p04) {};
    \draw(0,3)      node[dot]   (p04) {};
    \draw(0.5,0)      node[dot]   (p10) {};
    \draw(0.5,0.5)      node[dot]   (p11) {};
    \draw(0.5,1)      node[dot]   (p12) {};
    \draw(0.5,1.5)      node[dot]   (p13) {};
    \draw(0.5,2)      node[dot]   (p13) {};
    \draw(0.5,2.5)      node[dot]   (p13) {};
    \draw(1,0)      node[dot]   (p20) {};
    \draw(1,0.5)      node[dot]   (p21) {};
    \draw(1,1)      node[dot]   (p22) {};
    \draw(1,1.5)      node[dot]   (p22) {};
    \draw(1,2)      node[dot]   (p22) {};
    \draw(1.5,0)      node[dot]   (p30) {};
    \draw(1.5,0.5)      node[dot]   (p31) {};
    \draw(1.5,1)      node[dot]   (p31) {};
    \draw(1.5,1.5)      node[dot]   (p31) {};
    \draw(2,0)      node[dot]   (p40) {};
    \draw(2,0.5)      node[dot]   (p40) {};
    \draw(2,1)      node[dot]   (p40) {};
    \draw(2.5,0)      node[dot]   (p40) {};
    \draw(2.5,0.5)      node[dot]   (p40) {};
 
    \draw(3,0)      node[dot]   (p40) {};
    \draw [->] (0,0) -- (0,3.5);
    \draw [->] (0,0) -- (3.5,0);
\end{scope}  

\end{tikzpicture}%
\caption{Minkowski sum: $\trgset{2}{3}+\trgset{2}{3} = \trgset{2}{6}$.}%
\label{fig:msum}
\end{figure}
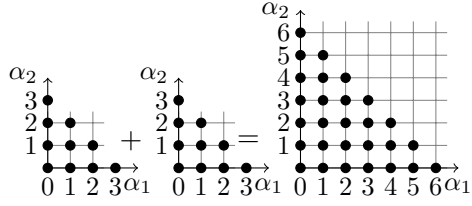

Finally, for  $\set{A} \in \bbN^{\nvar}$, we define its extension\footnote{In the papers \cite{Brachat.etal10LAaiA-Symmetric}, \cite{Laurent.Mourrain09AdM-generalized} and others this notation was used for the set of monomials.} as
\[
\set{A}^{+} \eqdef \set{A} \cup (\set{A} + \{\vect{e}_1\}) \cup \cdots \cup (\set{A} + \{\vect{e}_{\nvar}\}),
\]
where $\vect{e}_{k}$, $1\le k\le\nvar$,  denotes the vector of $\mathbb{N}^{\nvar}$ having a 1 in the $k$-th entry and zeros elsewhere. The exterior boundary of $\set{A}$ is defined as
\[
\delta\set{A} = \set{A}^{+} \setminus \set{A}.
\]
For example, for $\set{A} = \trgset{\nvar}{d}$ we have that
\[
\set{A}^{+} = \trgset{\nvar}{d+1}, \quad \delta\set{A} = \degset{\nvar}{d+1}.
\]

\subsection{Quasi-Hankel and quasi-Vandermonde matrices}
Let $\set{A}$ be a multi-index set ordered according to $\prec $:
\begin{equation}\label{eq:mind_set}
\set{A} = \{ \vect{\alpha}_1, \ldots, \vect{\alpha}_{M}\},\quad \vect{\alpha}_{k} \in \bbN^{\nvar}
\end{equation}
 and $\arr{h} = \{h_{\vect{\alpha}}\}_{\alpha \in \set{D}} \subset \mathbb{C}$ be an array of numbers, where $2\set{A}\subset \set{D}\subset \bbN^{\nvar}$.
Then the symmetric quasi-Hankel  \cite{Mourrain.Pan00Joc-Multivariate} matrix is defined as
\[
\hmatr{\set{A}}{\arr{h}} \eqdef [ h_{\vect{\alpha}_{i} + \vect{\alpha}_{j}} ]_{i,j=1}^{M,M}
\]

\begin{example}
For $\set{A} = \trgset{1}{d} = \{0, \ldots, d\}$, $M = d+1$, 
the quasi-Hankel matrix is the ordinary Hankel matrix 
\begin{equation}\label{eq:hankel}
\hmatr{\set{A}}{\iarr{h}}  = [ h_{k+l} ]_{k,l=0}^{d,d} = 
\left[
\begin{mediumarray}{cccc}
\cellcolor{lightgray}h_0 & \cellcolor{lightgray}h_1 & \cellcolor{lightgray} \cdots  & \cellcolor{lightgray}h_d \\
\cellcolor{lightgray}h_1 & \cellcolor{lightgray}h_2 &  \cellcolor{lightgray}\iddots & h_{d+1} \\
\cellcolor{lightgray}\vdots &   \cellcolor{lightgray} \iddots &  \iddots & \vdots \\
\cellcolor{lightgray}h_d & h_{d+1}  &  \cdots & h_{2d} 
\end{mediumarray}
\right].
\end{equation}
\end{example}

\begin{example}
In case $m > 1$, for example $\set{A} = \trgset{2}{3}$ 
(as in Fig.~\ref{fig:msum}) the matrix $\hmatr{\set{A}}{\iarr{h}} $ has the form 
\begin{equation}\label{eq:qhankel_trgset}
\left[
\begin{smallarray}{c|cc|ccc|cccc}
\cellcolor{lightgray}h_{0,0} &\cellcolor{lightgray} h_{1,0} & \cellcolor{lightgray} h_{0,1} &  \cellcolor{lightgray} h_{2,0} & \cellcolor{lightgray} h_{1,1} & \cellcolor{lightgray} h_{0,2} & \cellcolor{lightgray} h_{3,0} & \cellcolor{lightgray} h_{2,1} & \cellcolor{lightgray} h_{1,2} & \cellcolor{lightgray} h_{0,3} \\
\hline
\cellcolor{lightgray} h_{1,0} & \cellcolor{lightgray} h_{2,0} & \cellcolor{lightgray} h_{1,1} &  \cellcolor{lightgray} h_{3,0} & \cellcolor{lightgray} h_{2,1} & \cellcolor{lightgray} h_{1,2} & h_{4,0} & h_{3,1} & h_{2,2} & h_{1,3} \\
\cellcolor{lightgray} h_{0,1} & \cellcolor{lightgray} h_{1,1} & \cellcolor{lightgray} h_{0,2} &  \cellcolor{lightgray} h_{2,1} & \cellcolor{lightgray} h_{1,2} & \cellcolor{lightgray} h_{0,3} & h_{3,1} & h_{2,2} & h_{1,3} & h_{0,4} \\
\hline
\cellcolor{lightgray} h_{2,0} & \cellcolor{lightgray} h_{3,0} & \cellcolor{lightgray} h_{2,1} &  h_{4,0} & h_{3,1} & h_{2,2} & h_{5,0} & h_{4,1} & h_{3,2} & h_{2,3} \\
\cellcolor{lightgray} h_{1,1} & \cellcolor{lightgray} h_{2,1} & \cellcolor{lightgray} h_{1,2} &  h_{3,1} & h_{2,2} & h_{1,3} & h_{4,0} & h_{3,2} & h_{2,3} & h_{1,4} \\
\cellcolor{lightgray} h_{0,2} & \cellcolor{lightgray} h_{1,2} & \cellcolor{lightgray} h_{0,3} &  h_{2,2} & h_{1,3} & h_{0,4} & h_{3,2} & h_{2,3} & h_{1,4} & h_{0,5} \\
\hline
\cellcolor{lightgray} h_{3,0} & h_{4,0} & h_{3,1} &  h_{5,0} & h_{4,1} & h_{3,2} & h_{6,0} & h_{5,1} & h_{4,2} & h_{3,3} \\
\cellcolor{lightgray} h_{2,1} & h_{3,1} & h_{2,2} &  h_{4,1} & h_{3,2} & h_{2,3} & h_{5,1} & h_{4,2} & h_{3,3} & h_{2,4} \\
\cellcolor{lightgray} h_{1,2} & h_{2,2} & h_{1,3} &  h_{3,2} & h_{2,3} & h_{1,4} & h_{4,2} & h_{3,3} & h_{2,4} & h_{1,5} \\
\cellcolor{lightgray} h_{0,3} & h_{1,3} & h_{0,4} &  h_{2,3} & h_{1,4} & h_{0,5} & h_{3,3} & h_{2,4} & h_{1,5} & h_{0,6} 
\end{smallarray}
\right],
\end{equation}
where we omitted parentheses in subscripts for conciseness. 
Note that in \eqref{eq:qhankel_trgset}, each block corresponds to a subset $\degset{2}{k}$, $0 \le k \le 6$.
\end{example}

Let $\set{A} \subset \bbZp^{\nvar}$ be the ordered multi-index set defined in \eqref{eq:mind_set},
and $\vect{z}_{1}, \ldots, \vect{z}_{r} \in \bbC^{\nvar}$ be a set of points.
Then the quasi-Vandermonde matrix is defined as
\[
\vmatr{\set{A}}{\vect{z}_{1}, \ldots, \vect{z}_{r}} \eqdef \bmx (\vect{z}_{j})^{\vect{\alpha}_{i}} \emx_{i,j=1}^{N,r}.
\]
In particular, for $m=1$, $\set{A} = \trgset{1}{d} = \{0,\ldots,d\}$, it is just an ordinary rectangular Vandermonde matrix.

\begin{example}
For $\set{A} = \trgset{2}{3}$, $r=3$ and $\vect{z}_k = \bsm \lambda_k \\ \mu_k\esm$, $k\in\{1,2,3\}$,  we have that $\vmatr{\set{A}}{\vect{z}_{1}, \vect{z}_{2}, \vect{z}_{3}} = $
\small\begin{equation}\label{eq:vander_trgset}
\left[\!
\begin{array}{c@{\;}|c@{\;\,}c@{\;\,}|c@{\;\,}c@{\;\,}c@{\;\,}|c@{\;\,}c@{\;\,}c@{\;\,}c@{\;}}
1 & 
\lambda_1 & 
\mu_1 & 
 \lambda_1^2 &
  \lambda_1 \mu_1 &
   \mu_1^2 &
\lambda_1^3 & 
\lambda_1^2 \mu_1 & 
\lambda_1 \mu_1^2  & 
\mu_1^3  \\
1 & 
\lambda_2 & 
\mu_2 & 
 \lambda_2^2 &
  \lambda_2 \mu_2 &
   \mu_2^2 &
\lambda_2^3 & 
\lambda_2^2 \mu_2 & 
\lambda_2 \mu_2^2  & 
\mu_2^3  \\
1 & 
\lambda_3 & 
\mu_3 & 
 \lambda_3^2 &
  \lambda_3 \mu_3 &
   \mu_3^2 &
\lambda_3^3 & 
\lambda_3^2 \mu_3 & 
\lambda_3 \mu_3^2  & 
\mu_3^3  
\end{array}
\right]^{\T}\!\!\!.
\end{equation}\normalsize
\end{example}

Finally, we will use the following definition.
\begin{definition}\label{def:A_independence}
The points $\vect{z}_1, \ldots, \vect{z}_r \in \bbC^{\nvar}$ are called $\set{A}$-independent if 
\[
\rrank \vmatr{\set{A}}{\vect{z}_{1}, \ldots, \vect{z}_{r}} = r.
\]
\end{definition}
The notion of $\set{A}$-independence is equivalent to the fact the monomials $\{\vect{x}^{\vect{\alpha}}\}_{\vect{\alpha}\in \set{A}}$   taken on the grid of points $\{\vect{z}_1, \ldots, \vect{z}_r\}$ form a set of $\#{\set{A}}$ vectors spanning a linear space of dimension $r$. Hence, these monomials can interpolate any function on this grid. 

If $m=1$, the points are always $\set{A}$-independent for any $\set{A}$ with $r \le \#\set{A}$.
Although for $m > 1$ this is no longer the case,  the following remark holds true.

\begin{remark}\label{rem:A_independence_generic}
$\set{A}$-independence is a generic property, 
i.e., points $\{\vect{z}_j\}^r_{j=1}$ randomly drawn according to an absolutely continuous distribution are $\set{A}$-independent if $r \le \#\set{A}$. 
This follows from the fact that the set of $\set{A}$-independent points is open in the Zariski topology.
\end{remark}

\section{Hankel and Quasi-Hankel matrix completion}\label{seq:qh_completion}
In this paper, we consider the following  problem.
\begin{problem}\label{prob:qh_completion}
Given $\calA = \trgset{m}{d}$, and  $\{h_{\vect{\alpha}}\}_{\vect{\alpha} \in \calA}$,
\begin{equation}\label{eq:qhankel_rankmin}
\minim_{h_{\vect{\alpha}}, \vect{\alpha} \in 2 \calA \setminus \calA} \rrank \hmatr{\set{A}}{\arr{h}},
\end{equation} 
where $\arr{h} = \{h_{\vect{\alpha}}\}_{\vect{\alpha} \in 2\set{A}}$.
\end{problem}
If $m=1$, we have $\calA = \{0,\ldots,d\}$ (the matrix $\scrH_{\calA} (\arr{h})$ shown in \eqref{eq:hankel});
in this case, only the values $h_0,\ldots, h_{d}$ are known (shown in gray in \eqref{eq:hankel}) and $h_{d+1},\ldots,h_{2d}$ are to be completed.
The matrix completion for Hankel matrices is entirely solved \cite{Heinig.Rost84-Algebraic}, and its solution is related to the Sylvester's algorithm for CP decomposition of $2\times \cdots \times 2$ tensors \cite{Iarobbino.Kanev99-Power}. 
We recall the Hankel case in Section~\ref{sec:lin_rec}.

In the general case ($m > 1$), the upper block-triangular part of the quasi-Hankel matrix is known (\textit{e.g.,} in \eqref{eq:qhankel_trgset} it is shown in gray).
As it was shown in \cite{Brachat.etal10LAaiA-Symmetric},  symmetric  CP decomposition can be reduced to solving Problem~\ref{prob:qh_completion}.
The general quasi-Hankel matrix completion problem  \eqref{eq:qhankel_rankmin} can be solved in some  cases when the rank is sufficiently small (this is exactly the case when the CP tensor decomposition can be solved without considering incomplete quasi-Hankel matrices, \textit{e.g.}, using the Catalecticant algorithm \cite{Iarobbino.Kanev99-Power,Oeding.Ottaviani13JoSC-Eigenvectors}).
We recall these cases in Section~\ref{sec:qh_rank_minim}.

\subsection{Solution of Hankel low-rank completion}\label{sec:lin_rec}
Probably the first complete solution of Problem~\ref{prob:qh_completion} for $m=1$ is  contained in \cite[Ch. II]{Iohvidov82-Toeplitz}, where the problem of matrix completion is called ``singular extension'' of Hankel matrices.
In this section, we will use the theory from \cite{Heinig.Rost84-Algebraic}, and we will provide a summary of results in a simplified form.

For a finite sequence $\vect{h}_d = \bmx h_0 & \ldots& h_{d}\emx^{\T} \in \bbC^{d+1}$, we denote by $\hrank{\vect{h}_d}$ the smallest number $r$ such that there exists a non-zero vector $\vect{q} = \bmx q_0 & q_1&\ldots &q_r \emx^{\T} \neq 0$ for which
\begin{equation}\label{eq:linrec}
q_0 h_{k}  + \ldots + q_{r} h_{k+r}=0, \quad \forall k\in\{0,\ldots, d-r\}.
\end{equation}
The value $\hrank{\vect{h}_d}$ is equal to the maximal rank of the submatrices of the known triangular part of the Hankel matrix \eqref{eq:hankel} (or, equivalently to the maximal rank of a Hankel matrix that can be constructed only from the elements of the vector $\vect{h}_d$). 
The number $\hrank{\vect{h}_d}$ is called \emph{the first characteristic degree} of $\vect{h}_d$ \cite[Def. 5.3, page 81]{Heinig.Rost84-Algebraic}.
We will also call the vector  $\vect{q}$ as \emph{the characteristic vector}\footnote{In \cite{Heinig.Rost84-Algebraic}, $\vect{q}$ does not have a specific name and is denoted by  $p$.} of $\vect{h}_d$. Next, we recall the following basic properties of ${\vect{h}_d}$ and $\vect{q}$.
\begin{itemize}
\item For any $\vect{h}_d \in \bbC^{d+1}$, $\hrank{\vect{h}_d} \le \frac{d+2}{2}$ (see \cite[Prop. 5.4]{Heinig.Rost84-Algebraic}).
\item If, in addition, $\hrank{\vect{h}_d} < \frac{d+2}{2}$, then the characteristic vector $\vect{q}$ is unique up to scaling by a constant (see \cite[page 84]{Heinig.Rost84-Algebraic}).
\end{itemize}
In what follows, we assume that $q_r \neq 0$ (this is a generic case). The treatment of the case $q_r = 0$ can be found in \cite[page 84]{Heinig.Rost84-Algebraic}).
Then an explicit representation of ${\vect{h}_d}$ is determined by the \emph{characteristic polynomial} 
\begin{equation}\label{eq:char_poly}
q(z) = \sum\limits_{j=0}^r q_j z^j  = q_r \sum\limits_{k=1}^s (z - \lambda_k)^{\nu_k},
\end{equation}
where  $\lambda_k \in \bbC$ are distinct,  $\nu_k$ are the multiplicities of the roots, and $r = \nu_1 + \cdots + \nu_s$.
\begin{proposition}[A special case of {\cite[Thm 8.1]{Heinig.Rost84-Algebraic}}]
If $q_r \neq 0$ and all the roots of $q(z)$ are simple ($\nu_k = 1$, $s=r$), then the sequence $\vect{h}_d$ admits a representation
\begin{equation}\label{eq:simple_roots}
h_k = \sum\limits_{j=1}^r c_{j} \lambda_j^k.
\end{equation}
\end{proposition}
\begin{example}
For $\vect{h}_d = \bmx 1 & \lambda & \cdots & \lambda^d \emx^{\T}$ (see Theorem~\ref{thm:rank_one}), we have $\hrank{\vect{h}_d}  = 1$ and the characteristic vector can be chosen as  $\vect{q} = \bmx -\lambda & 1 \emx^{\T}$.
\end{example}
The representation \eqref{eq:simple_roots}  is called a \emph{canonical representation} of $\vect{h}_d$ (or the unique canonical representation in case $\hrank{\vect{h}_d} < \frac{d+2}{2}$).
In the general case \eqref{eq:char_poly}, there are multiple roots and the canonical representation has a more complicated form, as stated by the proposition below.

\begin{proposition}[A special case of {\cite[Thm 8.1]{Heinig.Rost84-Algebraic}}]\label{prop:canonical_hankel_general}
Let  $q_r \neq 0$, $\lambda_1 = 0$ and $\lambda_k \neq 0$, and the multiplicities $\nu_k$ of the roots $\lambda_k$ in \eqref{eq:char_poly} are such that $\nu_1 \ge 0$ and $\nu_k > 1$ for $k \ge 2$. Then the sequence $\vect{h}_d$ has the form
\begin{equation}\label{eq:multiple_roots}
h_k = \sum\limits_{j=2}^s c_j(k) \lambda_j^k + \sum\limits_{l=0}^{\nu_1-1} c_{1,l} \delta(k,l), 
\end{equation}
where $c_j(k)$ are polynomials of degree $(\nu_k-1)$ and $\delta(k,l)$ is the Kronecker symbol.
\end{proposition}
\begin{remark}
In the statement of Proposition~\ref{prop:canonical_hankel_general}, if $\nu_1=0$ then the second term in \eqref{eq:multiple_roots} is absent. 
\end{remark}
\begin{example}
In the extreme case $q(z) = q_r z^r$ (\textit{i.e.}, $r=\nu_1$), the canonical representation \eqref{eq:multiple_roots} becomes
\begin{equation}\label{eq:initial_zero_root}
\vect{h}_d = \bmx c_{1,0} & \cdots & c_{1,r-1} & 0 & \cdots & 0 \emx^{\T},
\end{equation}
\end{example}
Apart from the explicit form of the sequence $\vect{h}_d$, the characteristic polynomial gives a solution to the rank minimization problem \eqref{eq:qhankel_rankmin} (for $m=1$).
\begin{proposition}[{A corollary of \cite[Thm. 5.14]{Heinig.Rost84-Algebraic}}]
Let $\vect{q} \in \bbC^{r+1}$ be a characteristic vector of $\vect{h}_d$ with  $q_r  \neq 0$, $r = \hrank{\vect{h}_d}$. Then it holds that
\begin{enumeratex} 
\item The rank of the minimal rank completion of $\hmatr{\calA}{\iarr{h}}$ in \eqref{eq:hankel} is $r$.
\item A minimum rank completion $h_{d+1}, \ldots, h_{2d}$ is given by
the recursive continuation, for $k > d-r$:
\begin{equation}\label{eq:canonical_completion}
~\hspace{-3em} h_{k+r} = -\frac{1}{q_r} (q_0 h_{k}  + \cdots + q_{r-1} h_{k+r-1}) 
\end{equation}
In addition, the values of $h_{k+r}$ for $k > d-r$ can be obtained by using the corresponding formula of canonical representation, namely \eqref{eq:simple_roots} or \eqref{eq:multiple_roots}.
\item If $r < \frac{d+2}{2}$, then the minimum rank completion is unique.
\end{enumeratex}
\end{proposition}

\subsection{Solution of quasi-Hankel low-rank completion}\label{sec:qh_rank_minim}
Unlike Hankel matrices, for $m > 1$ the quasi-Hankel completion problem does not have a closed form solution.
The simplest case when the solution can be easily found is given by a flat extension theorem {\cite[Thm. 1.4]{Laurent.Mourrain09AdM-generalized}}.
\begin{proposition}[Corollary of {\cite[Thm. 1.4]{Laurent.Mourrain09AdM-generalized}}]\label{prop:flat_ext}
Let $\set{A} = \trgset{\nvar}{d}$, and the values $\{{h}_{\vect{\alpha}}\}_{\vect{\alpha} \in \set{A}}$ be given.
Moreover, let $d' \eqdef \lfloor \frac{d}{2} \rfloor$, $\set{B} \eqdef \trgset{\nvar}{d'-1}$, and it holds that 
\[
\rank{\hmatr{\set{B}}{\arr{h}}} = \rank{\hmatr{\set{B}^{+}}{\arr{h}}} = r.
\]
Then the unique solution of Problem~\ref{prob:qh_completion} has rank $r$.
\end{proposition}
Proposition~\ref{prop:flat_ext} induces bounds on rank that are too restrictive. For example, for $\set{A} = \trgset{2}{3}$, $d'=1$ and only rank-one cases can be treated. 
In what follows, we consider arrays of special form, for which Proposition~\ref{prop:flat_ext} can be extended. 

Let $\set{C} \subset \bbN^{\nvar}$ be arbitrary, $\vect{z}_{1}, \ldots, \vect{z}_{r} \in \bbC^{\nvar}$,
$c_1, \ldots, c_r \in \bbC$ be some coefficients, and the array of coefficients $\iarr{h} = \{{h}_{\vect{\alpha}}\}_{\vect{\alpha} \in 2\set{C}}$ be defined as
\begin{equation}\label{eq:exp_array}
h_{\vect{\alpha}} = \sum\limits_{k=1}^r c_k \vect{z}_k^{\vect{\alpha}}.
\end{equation}
It is easy to see that the following lemma holds true

\begin{lemma}\label{lem:rank_exp_array}
For $\arr{h}$ defined as in \eqref{eq:exp_array}, the quasi-Hankel matrix admits the factorization $\hmatr{\set{C}}{\arr{h}} =$
\begin{equation}\label{eq:vandermonde_fac}
 \vmatr{\set{C}}{\vect{z}_{1}, .., \vect{z}_{r}}
\diag{c_1, \ldots, c_r}
(\vmatr{\set{C}}{\vect{z}_{1}, .., \vect{z}_{r}})^{\top}
\end{equation}
and $\rank{\hmatr{\set{C}}{\arr{h}}} \le r$.
If, in addition, $c_k \neq 0$ for all $k$, and  $\vect{z}_1, \ldots, \vect{z}_r$  are $\set{C}$-independent, then 
\[
\rank{\hmatr{\set{C}}{\arr{h}}} = r.
\]
\end{lemma}

Next, we describe an easy improvement of Proposition~\ref{prop:flat_ext} for arrays of form \eqref{eq:exp_array}.

\begin{proposition}\label{prop:flat_ext_exp_array}
Let $\set{A} = \trgset{\nvar}{d}$, $d' \eqdef \lfloor \frac{d}{2} \rfloor$,~$\set{B} \eqdef \trgset{\nvar}{d'}$ (it is easy to see that $2\set{B} \subset \set{A}$).
Assume that the coefficients $\{h_{\vect{\alpha}}\}_{\vect{\alpha} \in \set{A}}$ have the form \eqref{eq:exp_array}, where the points $\vect{z}_1,\ldots,\vect{z}_{r}$ are $\set{B}$-independent and
$c_1,\ldots,c_r$ are nonzero.
Then, the following hold true
\begin{enumeratex}
\item The rank of the minimal completion in Problem~\ref{prob:qh_completion} is $r$. A minimal rank completion is  obtained by defining $\{h_{\vect{\alpha}}\}_{\vect{\alpha} \in 2\set{A}}$ using the same formula  \eqref{eq:exp_array} (called canonical completion).
\item If $d$ is odd, the canonical completion  is unique.
\item If $d$ is even, and $\vect{z}_1, \ldots, \vect{z}_r$ are $\trgset{\nvar}{d'-1}$-independent, then the
completion is unique (the canonical completion of Eq. \eqref{eq:exp_array} is the only possible).
\end{enumeratex}
\end{proposition}
\proofapp

For example, for the completion of structure shown in \eqref{eq:qhankel_trgset} ($m=2$ and $d=3$), Proposition~\ref{prop:flat_ext_exp_array} covers cases up to rank $3$.
Indeed, let  $\vect{z}_k = \bsm \lambda_k \\ \mu_k\esm$, $k\in\{1,2,3\}$,  we have that if the Vandermonde matrix has full rank \eqref{eq:vander_trgset}, then the matrix completion is unique and is given by the canonical completion.

\begin{remark}\label{rem:generic_qhankel_completion}
If the points $\vect{z}_1, \ldots, \vect{z}_r$ are generic, due to Remark~\ref{rem:A_independence_generic}, we can replace the $\set{B}$-independence and $\trgset{\nvar}{d'-1}$-independence by the bounds $r \le \binom{d'+m}{m}$ and $r \le \binom{d'+m-1}{m}$ respectively.
\end{remark}

Finally, we note that for generic cases, the rank bound in Remark~\ref{rem:generic_qhankel_completion} can be further improved by using results of \cite[Theorem~{2.4}]{Oeding.Ottaviani13JoSC-Eigenvectors} (since uniqueness of matrix completion is related to uniqueness of tensor decomposition).
For our purposes, however, Proposition~\ref{prop:flat_ext_exp_array} will be sufficient.


\section{Optimality conditions of nuclear norm minimization}\label{sec:nnmin_genprop}
\subsection{Optimality conditions in convex optimization}\label{sec:cvx_opt_cond}
We recall some definitions from the field of convex optimization \cite[Chapter D]{Hiriart-Urruty.Lemarechal01-Fundamentals}.
For a convex (possibly non-differentiable) function $f: \bbR^{N} \to \bbR$, the subdifferential \cite[Def. 1.2.1]{Hiriart-Urruty.Lemarechal01-Fundamentals}  of $f$ is defined as the set $\partial f(\vect{x}) \subset \bbR^{N}$
\[
\partial f(\vect{x}) \eqdef \{\vect{d} : f({\vect{y}}) - f(\vect{x}) \ge \langle \vect{d}, \vect{y}-\vect{x}\rangle\quad \forall \vect{y} \in \bbR^{N} \}.
\]
In particular, if $f$ is differentiable at a point $\vect{x}$, then the subgradient has only one element: the usual gradient, i.e., $\partial f(\vect{x}) = \{ \nabla f(\vect{x})\}$.

For the unconstrained  convex optimization problem
\begin{equation}\label{eq:cvx_opt}
\min_{\vect{x} \in \bbR^{N}}  f(\vect{x}),
\end{equation}
we recall the first-order optimality condition, which is necessary and sufficient in this case.
\begin{lemma}[First-order optimality, {\cite[Thm. 2.2.1]{Hiriart-Urruty.Lemarechal01-Fundamentals}}]\label{lem:cvx_opt_first_order}
A  point $\vect{x}$ is a minimum point of \eqref{eq:cvx_opt} if and only if 
\begin{equation}\label{eq:first_order_condition}
\vect{0} \in  \partial f(\vect{x}).
\end{equation}
\end{lemma}
In particular, for differentiable functions, the condition of Lemma~\ref{lem:cvx_opt_first_order} reduces to $\nabla f(\vect{x}) = \vect{0}$.
Next, we use a simple uniqueness condition for the minimizer.
\begin{lemma}[Sufficient condition of uniqueness]\label{lem:cvx_opt_uniqueness}
A  point $\vect{x}$ is the unique minimizer of \eqref{eq:cvx_opt} if 
\begin{equation}\label{eq:uniqueness_convex}
\vect{0} \in \inter{\partial f(\vect{x})},
\end{equation}
where $\inter{\cdot}$ denotes the interior of a set.
\end{lemma}
\begin{proof}
In this case for any $\vect{y}\neq \vect{x}$ there exists $\delta > 0$ such that $\delta\cdot(\vect{y}-\vect{x}) \in \partial f(\vect{x})$.
By definition of $\partial f(\vect{x})$, we have that $f(\vect{y}) -f(\vect{x}) \ge \delta \|\vect{y}-\vect{x}\|^2_2 > 0$.
\end{proof}

\subsection{Subdifferential of the nuclear norm}
First, we recall the form of the subdifferential of the nuclear norm of a matrix. 
Let $\matr{X} \in \bbR^{n_1 \times n_2}$ be a matrix of rank $r$, and let $\matr{X} = \matr{U} \matr{\Sigma} \matr{V}^{\T}$ be an SVD of $\matr{X}$, 
where $\matr{U} \in \bbR^{n_1 \times r}$, $\matr{V} \in \bbR^{n_2\times r}$ and $\matr{\Sigma} \in \bbR^{r\times r}$ is a diagonal matrix of nonzero singular values. 
Next, let $\matr{U}_{\bot} \in \bbR^{n_1 \times (n_1-r)}$, $\matr{V}_{\bot} \in \bbR^{n_2\times (n_2-r)}$ be the bases of the left and right nullspaces of $\matr{X}$, respectively.
Then, according to \cite[p. 41]{Watson92LAaiA-Characterization}, the subdifferential of the nuclear norm at $\matr{X}$ is equal to $\partial\|\matr{X}\|_{*} = $
\[
\{\matr{B} + \matr{U}_{\bot} \matr{W} \matr{V}^{\top}_{\bot}: \matr{W} \in \bbR^{(n_1-r) \times (n_2-r)}\!, \|\matr{W}\|_2 \le 1\},
\]
where $\matr{B}$ is defined as
\begin{equation}\label{eq:skew_projector}
\matr{B} \eqdef \matr{U} \matr{V}^{\T},
\end{equation}
and $\|\cdot\|_2$ is the spectral norm (largest singular value).

\begin{remark}\label{rem:B_uniqueness}
When the SVD of $\matr{X}$ is not unique, then the matrix $\matr{B}$ from \eqref{eq:skew_projector} is still defined uniquely. In fact, if $\matr{X}$ has multiple singular values, its matrix of left (resp. right) singular vectors is of the form $\matr{U}\matr{\Theta}$ (resp. $\matr{V}\matr{\Theta}$), where $\matr{\Theta}$ is an orthonormal matrix which commutes with $\matr{\Sigma}$. 
\end{remark}
Next we consider a real nuclear norm minimization problem
\begin{equation}\label{eq:nnmin_real}
\widehat{\vect{p}} = \arg \min_{\vect{p} \in \bbR^{N}} f(\vect{p}), \quad f(\vect{p}) \eqdef \| \struct{\vect{p}} \|_*,
\end{equation}
where $\scrS$ is defined in \eqref{eq:affine_structure}. By an analogue of chain rule \cite[Thm. 4.2.1]{Hiriart-Urruty.Lemarechal01-Fundamentals}, we immediately have that
\begin{eqnarray}\label{eq:subgradient_Sp}
\partial f (\vect{p}) &=& 
\Big\{
\bmx \langle \matr{S}_1, \matr{H}\rangle_F & \cdots & \langle \matr{S}_N, \matr{H}\rangle_F \emx^{\T} : \nonumber\\
&& \phantom{\Big\{}\matr{H} \in \partial\|\matr{X}\|_{*}, \matr{X} = \struct{\vect{p}}  \Big\}, 
\end{eqnarray}
where $\langle\cdot,\cdot\rangle_F$ denotes the Frobenius inner product.
In the following section, we will give a more compact form or subdifferentials and optimality conditions.

\subsection{Real nuclear norm minimization}
First, we assume that $\matr{X}$ is symmetric and $n_1 = n_2 = n$, then we have that $\matr{U}_{\bot}\matr{U}_{\bot}^{\top} = \matr{V}_{\bot}\matr{V}_{\bot}^{\top} = \matr{Q}$, where $\matr{Q}$ is the projector on the nullspace of $\matr{X}$. Then the subdifferential can be rewritten as
\begin{equation}\label{eq:subgradient_nnorm}
\partial\|\matr{X}\|_{*} = \{ \matr{B} + \matr{Q}\matr{M}\matr{Q}: \|\matr{M}\|_2 \le 1\},
\end{equation}
where $\matr{B}$ is defined as in \eqref{eq:skew_projector}. We note that $\matr{Q} = \matr{I} -\matr{P}$, where $\matr{P}$ is the orthogonal projector on the column space of $\matr{X}$, which can be obtained as
\begin{equation}\label{eq:projector}
\matr{P} \eqdef \matr{U}\matr{U}^{\T} =\matr{B} \matr{B}^{\T}.
\end{equation}
Next, we define the matrix 
\begin{equation}\label{eq:bfG}
\bfS \eqdef \bmx \mvec(\matr{S}_1) & \ldots &  \mvec(\matr{S}_{N})\emx,
\end{equation}
and for a matrix $\matr{P} \in \bbR^{n\times n}$ we define
\begin{equation}\label{eq:fo_condition_matrix_left}
\mathscr{A}(\matr{P}) \eqdef \bfS^{\T} ((\matr{I} - \matr{P}) \otimes (\matr{I} - \matr{P})) \in \bbR^{N \times n^2},
\end{equation}
where $\otimes$ denotes the Kronecker product.
Then, by \eqref{eq:subgradient_Sp}, the subgradient of $f$  can be rewritten  as $\partial f (\vect{p})=$
\begin{equation}\label{eq:subgradient_Sp_matrix}
\left\{\bfS^{\T}\! \mvec(\matr{B}) + \mathscr{A}(\matr{P}) \mvec (\matr{M})\!:
 \|\matr{M}\|_2 \le 1 \right\}.
\end{equation}
with $\matr{M}\!\in\bbR^{n\times n}$.
Using \eqref{eq:subgradient_Sp_matrix}, we can rewrite Lemma~\ref{lem:cvx_opt_first_order} and Lemma~\ref{lem:cvx_opt_uniqueness} as follows.
\begin{lemma}[First-order  optimality conditions for real nuclear norm minimization]\label{lem:opt_first_order_nnmin}
A point $\vect{p}^{*}$ is a minimizer of \eqref{eq:nnmin_real} if and only if there exists a matrix $\matr{M} \in \bbR^{n\times n}$ with $\|\matr{M}\|_2 \le 1$ such that
\begin{equation}\label{eq:fo_condition_matrix}
\mathscr{A}(\matr{P}) \mvec (\matr{M}) = -\bfS^{\T} \mvec(\matr{B}).
\end{equation}
where $\matr{B}$ is as in \eqref{eq:skew_projector},
 $\matr{P}$ is as in \eqref{eq:projector}, and $\matr{Q} \eqdef \matr{I}- \matr{P}$.
\end{lemma}
\begin{proof}
Follows from \eqref{eq:subgradient_Sp_matrix} and Lemma~\ref{lem:cvx_opt_first_order}.
\end{proof}
\begin{lemma}[Sufficient condition of uniqueness]\label{lem:unique_minimizer_nnmin}
A point $\vect{p}$ is the unique minimizer of \eqref{eq:nnmin} if there exists $\matr{M} \in \bbR^{n\times n}$ with $\|\matr{M}\|_2 < 1$ such that \eqref{eq:fo_condition_matrix} holds and the matrix $\mathscr{A}(\matr{P})$ is full row rank (\ie, is of rank $N$).
\end{lemma}
\begin{proof}
Follows from \eqref{eq:subgradient_Sp_matrix} and Lemma~\ref{lem:cvx_opt_uniqueness}.
\end{proof}

\subsection{Complex nuclear norm minimization}
Now let $\struct{\vect{p}}$  be the matrix structure given in \eqref{eq:affine_structure}, but now $\matr{S}_0 \in \bbC^{n \times n}$, $\matr{S}_k \in \bbR^{n\times n}$ for  $k \ge 1$, and the parameter is complex $\vect{p}_{\bbC} \in \bbC^{N}$, \ie,
\eqref{eq:affine_structure} defines a complex-valued map $\scrS: \bbC^{N} \to \bbC^{n \times n}$.

In order to derive the optimality conditions for the complex-valued case, we construct an extended real-valued structure $\scrS_{ext}: \bbR^{2N} \to \bbR^{2n\times 2n}$ as follows.
For $\vect{p}_{ext} = \bmx \vect{p}_{\calR} \\ \vect{p}_{\calI} \emx$, where $\vect{p}_{\calR},\vect{p}_{\calI} \in \bbR^{N}$, we define an equivalent (non-symmetric) matrix structure:
\begin{equation}\label{eq:affine_ext0}
\scrS_{ext}(\vect{p}_{ext}) \eqdef 
\bmx \struct{\vect{p}_{\calR}} & -\struct{\vect{p}_{\calI}} \\ \struct{\vect{p}_{\calI}} &  \struct{\vect{p}_{\calR}} \emx
\end{equation}
From \cite{Day.Heroux01SJoSC-Solving}, the singular values of $\scrS_{ext}(\vect{p}_{ext})$ contain two copies of singular values of $\struct{\vect{p}_{\bbC}}$. Hence,
\[
\rank{\scrS_{ext}(\vect{p}_{ext})} = 2 \rank{\struct{\vect{p}_{\bbC}}},
\]
and the complex rank minimization problem \eqref{eq:rankmin}  is equivalent to real rank minimization 
\begin{equation}\label{eq:rankmin_ext}
\widetilde{\vect{p}}_{ext} =  \arg \min_{\vect{p}_{ext} \in \bbR^{2N}} \rrank \scrS_{ext}(\vect{p}_{ext}).
\end{equation}
Moreover, the problem \eqref{eq:nnmin}, is equivalent to the real nuclear norm minimization problem
\begin{equation}\label{eq:nnmin_ext}
\widehat{\vect{p}}_{ext} =  \arg \min_{\vect{p}_{ext} \in \bbR^{2N}} \| \scrS_{ext}(\vect{p}_{ext})\|_*.
\end{equation}
The following proposition shows that we can rewrite the optimality conditions for the problem \eqref{eq:nnmin_ext} in a convenient form, as in  Lemmae~\ref{lem:opt_first_order_nnmin} and \ref{lem:unique_minimizer_nnmin}.

\begin{proposition}[Complex optimality conditions]\label{prop:opt_sufficient_complex}
Let $\vect{p}_{\bbC} = \vect{p}_{\calR} + i \vect{p}_{\calI} \in \bbC^N$, $\struct{\vect{p}_{\bbC}} = \matr{U}_{\bbC} \Sigma \matr{V}_{\bbC}^{\H}$ be an SVD of the symmetric matrix $\struct{\vect{p}_{\bbC}}$, and define $\matr{B}_{\bbC} := \matr{U}_{\bbC}^{} \matr{V}_{\bbC}^{\H}$, $\matr{P}_{\bbC} = \matr{U}_{\bbC}^{} \matr{U}_{\bbC}^{\H}$, and $\matr{Q}_{\bbC} = \matr{I}_n - \matr{P}_{\bbC}$.
Then it holds that:

\begin{enumeratex}
\noindent\item The point  $\vect{p}_{\bbC}$ is a minimizer of \eqref{eq:nnmin_ext} if and only if there exists $\matr{M} \in \bbC^{n \times n}$ with $\|\matr{M}\|_2 \le 1$ that satisfies \eqref{eq:fo_condition_matrix}.

\noindent\item
The point $\vect{p}_{ext}$ is the unique minimizer of \eqref{eq:nnmin_ext} if there exists $\matr{M} \in \bbC^{n \times n}$ with $\|\matr{M}\|_2 < 1$ that satisfies \eqref{eq:fo_condition_matrix}, and $\mathscr{A}(\matr{P})$ is full row rank.
\end{enumeratex}
\end{proposition}
\proofapp

\begin{remark}\label{rem:equivalent_conditions_slrmc_DP15}
Lemmae~\ref{lem:opt_first_order_nnmin},~\ref{lem:unique_minimizer_nnmin}, and Proposition~\ref{prop:opt_sufficient_complex} are related to  optimality conditions used in \cite{Dai.Pelckmans15A-nuclear}. Indeed, the rank condition on $\rank{\mathscr{A}(\matr{P})} = N$ is equivalent to $\matr{Q}\matr{H}\matr{Q}^{\T} \neq 0$ for all $\matr{H}$ of the form 
\[
\forall \matr{H} =  \sum\limits_{k=1}^{N}\Delta p_k \matr{S}_k , \quad \Delta\vect{p} \in \bbC^{N} \setminus \{\vect{0}\},
\]
which corresponds to \cite[Proposition 2]{Dai.Pelckmans15A-nuclear}.
Next, the condition \eqref{eq:fo_condition_matrix} can be rewritten as
\[
\langle \matr{S}_k, \matr{B}+\matr{Q}\matr{M}\matr{Q}^{\T}\rangle_F = \matr{0},\quad \forall k\in\{1,\ldots,N\},
\]
which corresponds to the condition in \cite[Lemma 2]{Dai.Pelckmans15A-nuclear}.
The authors of \cite{Dai.Pelckmans15A-nuclear}, however, do not make explicit connection with general optimality conditions for convex optimization problems, presented in Section~\ref{sec:cvx_opt_cond}.
\end{remark}

\section{Main lemmas: optimality conditions and simple projectors}\label{sec:lemmas}

\subsection{Quasi-Hankel completion: basis matrices}\label{sec:basis_matrices}
First, we put the rank minimization problem  \eqref{eq:qhankel_rankmin} in the (standard) form \eqref{eq:rankmin}.
In order to do this, we explicitly write down the matrices $\matr{S}_k$ in  \eqref{eq:affine_structure}.
The constant part has the form
\[
\matr{S}_0 \eqdef \hmatr{\set{A}}{\arr{c}}, 
\]
where $\arr{c} = \{c_{\vect{\alpha}}\}_{\vect{\alpha} \in 2 \set{A}}$ is defined as
\[
{c}_{\vect{\alpha}} \eqdef 
\begin{cases}
{h}_{\vect{\alpha}}, & {\vect{\alpha}} \in \set{A},  \\
0, & {\vect{\alpha}} \in 2\set{A} \setminus \set{A}.
\end{cases}
\]

\begin{example}
In the case of $\set{A} = \trgset{2}{3}$ (depicted in \eqref{eq:qhankel_trgset}), we have that $\matr{S}_0 = \hmatr{\set{A}}{\arr{c}} = $
\small\[
\left[
\begin{array}{c|c@{\;}c@{\;}|c@{\;}c@{\;}c@{\;}|c@{\;}c@{\;}c@{\;}c}
h_{0,0} & h_{1,0} & h_{0,1} &  h_{2,0} & h_{1,1} & h_{0,2} & h_{3,0} & h_{2,1} & h_{1,2} & h_{0,3} \\
\hline
h_{1,0} & h_{2,0} & h_{1,1} &  h_{3,0} & h_{2,1} & h_{1,2} & 0         & 0         & 0         & 0         \\
h_{0,1} & h_{1,1} & h_{0,2} &  h_{2,1} & h_{1,2} & h_{0,3} & 0         & 0         & 0         & 0 \\
\hline
h_{2,0} & h_{3,0} & h_{2,1} &  0        & 0 & 0 & 0 & 0 & 0 & 0 \\
h_{1,1} & h_{2,1} & h_{1,2} &  0        & 0 & 0 & 0 & 0 & 0 & 0 \\
h_{0,2} & h_{1,2} & h_{0,3} &  0 & 0 & 0 &0 & 0 & 0 &0 \\
\hline
h_{3,0} & 0         & 0         &  0 & 0 & 0 & 0 & 0 & 0 & 0 \\
h_{2,1} & 0         & 0         &  0 & 0 & 0 & 0 & 0 & 0 & 0 \\
h_{1,2} & 0         & 0         &  0 & 0 & 0 & 0 & 0 & 0 & 0 \\
h_{0,3} & 0         & 0         &  0 & 0 & 0 & 0 & 0 & 0 & 0
\end{array}
\right].
\]\normalsize
In the case of Hankel matrices ($\set{A} = \trgset{1}{d}$, see \eqref{eq:hankel}), we have that $n = d+1$, and
$\matr{S}_0$ is given by
\begin{equation}\label{eq:hankel_G0}
\matr{S}_0 = \bmx
h_0     & h_1     & \cdots & h_{d-1} & h_{d} \\
h_1     & h_2     & \iddots & h_{d} & 0     \\
\vdots  & \iddots  & \iddots & \iddots  & \vdots  \\
h_{d-1} & h_{d} & \iddots &  & 0 \\
h_{d} & 0     & \cdots & 0 & 0
\emx
\end{equation}
\end{example}

The number of free parameters in \eqref{eq:affine_structure} is  $N = \#(2\set{A}) - \#(\set{A})$ (\textit{i.e.}, the number of unknowns in \eqref{eq:qhankel_rankmin}).
The corresponding set of multi-indices can be represented as
\begin{equation}\label{eq:ordering_missing_values}
2\set{A} \setminus \set{A}  = \{\bm{\beta}_1,\ldots,\bm{\beta}_{N}\},\quad \bm{\beta}_1 \prec \cdots \prec\bm{\beta}_{N}.
\end{equation}
Next, we define basis arrays $\iarr{e}^{\bm{\beta}} \eqdef \{e^{\bm{\beta}}\}_{\bm{\alpha} \in \bbZp^{\nvar}}$ as 
\[
e^{\bm{\beta}}_{\bm{\alpha}} \eqdef 
\begin{cases}
1, & {\bm{\alpha}} = {\bm{\beta}},  \\
0, & {\bm{\alpha}} \neq {\bm{\beta}},
\end{cases}
\]
and $\matr{S}_k$ as
\begin{equation}\label{eq:qhankel_Gk}
\matr{S}_k := \hmatr{\set{A}}{\iarr{e}^{\bm{\beta}_k}}.
\end{equation}
Finally, it is easy to see that
\[
\hmatr{\set{A}}{\iarr{h}} = \scrS(\vect{p}),
\]
if $\vect{p}$ is defined as  $p_k = \iarr{h}_{\bm{\beta}_k}$.

\begin{example}
For $\set{A} = \trgset{2}{3}$, shown in \eqref{eq:qhankel_trgset}, we have
\[
\begin{split}
\matr{S}_1 &= 
\left[
\begin{smallarray}{c|cc|ccc|cccc}
0  &  0 & 0 &   0 & 0 & 0 &   0 & 0 & 0 & 0 \\
\smallhline
0  &  0 & 0 &   0 & 0 & 0 &   1 & 0 & 0 & 0 \\
0  &  0 & 0 &   0 & 0 & 0 &   0 & 0 & 0 & 0 \\
\smallhline
0  &  0 & 0 &   1 & 0 & 0 &   0 & 0 & 0 & 0 \\
0  &  0 & 0 &   0 & 0 & 0 &   0 & 0 & 0 & 0 \\
0  &  0 & 0 &   0 & 0 & 0 &   0 & 0 & 0 & 0 \\
\smallhline
0  &  1 & 0 &   0 & 0 & 0 &   0 & 0 & 0 & 0 \\
0  &  0 & 0 &   0 & 0 & 0 &   0 & 0 & 0 & 0 \\
0  &  0 & 0 &   0 & 0 & 0 &   0 & 0 & 0 & 0 \\
0  &  0 & 0 &   0 & 0 & 0 &   0 & 0 & 0 & 0 
\end{smallarray}
\right],
\matr{S}_2 = 
\left[
\begin{smallarray}{c|cc|ccc|cccc}
0  &  0 & 0 &   0 & 0 & 0 &   0 & 0 & 0 & 0 \\
\smallhline
0  &  0 & 0 &   0 & 0 & 0 &   0 & 1 & 0 & 0 \\
0  &  0 & 0 &   0 & 0 & 0 &   1 & 0 & 0 & 0 \\
\smallhline
0  &  0 & 0 &   0 & 1 & 0 &   0 & 0 & 0 & 0 \\
0  &  0 & 0 &   1 & 0 & 0 &   0 & 0 & 0 & 0 \\
0  &  0 & 0 &   0 & 0 & 0 &   0 & 0 & 0 & 0 \\
\smallhline
0  &  0 & 1 &   0 & 0 & 0 &   0 & 0 & 0 & 0 \\
0  &  1 & 0 &   0 & 0 & 0 &   0 & 0 & 0 & 0 \\
0  &  0 & 0 &   0 & 0 & 0 &   0 & 0 & 0 & 0 \\
0  &  0 & 0 &   0 & 0 & 0 &   0 & 0 & 0 & 0 
\end{smallarray}
\right],  \\
\matr{S}_3 &= 
\left[
\begin{smallarray}{c|cc|ccc|cccc}
0  &  0 & 0 &   0 & 0 & 0 &   0 & 0 & 0 & 0 \\
\smallhline
0  &  0 & 0 &   0 & 0 & 0 &   0 & 0 & 1 & 0 \\
0  &  0 & 0 &   0 & 0 & 0 &   0 & 1 & 0 & 0 \\
\smallhline
0  &  0 & 0 &   0 & 0 & 1 &   0 & 0 & 0 & 0 \\
0  &  0 & 0 &   0 & 1 & 0 &   0 & 0 & 0 & 0 \\
0  &  0 & 0 &   1 & 0 & 0 &   0 & 0 & 0 & 0 \\
\smallhline
0  &  0 & 0 &   0 & 0 & 0 &   0 & 0 & 0 & 0 \\
0  &  0 & 1 &   0 & 0 & 0 &   0 & 0 & 0 & 0 \\
0  &  1 & 0 &   0 & 0 & 0 &   0 & 0 & 0 & 0 \\
0  &  0 & 0 &   0 & 0 & 0 &   0 & 0 & 0 & 0 
\end{smallarray}
\right], \ldots
\end{split}
\]
For Hankel structure, ($\set{A} = \trgset{1}{d}$, shown in \eqref{eq:hankel}), we have that $N=d$, and the matrices  $\matr{S}_k$, $k \in \{1,\ldots, N\}$ are defined as
\begin{equation}\label{eq:hankel_Gk}
\begin{split}
\matr{S}_1 &=\bsm
0    & 0     & \cdots & 0 & 0 \\
0     & 0     & \iddots & 0 & 1     \\
\vdots  & \iddots  & \iddots & \iddots  &  0 \\
0 & 0 & \iddots &  & \vdots\\
0 & 1  & 0 & \cdots & 0
\esm,
\;
\cdots,
\;
\matr{S}_{N-1} = \bsm
0    & 0     & \cdots & 0 & 0 \\
0     & 0     & \iddots & \iddots & \vdots      \\
\vdots  & \iddots  & \iddots & \iddots  & 0 \\
0 & \iddots & \iddots &  & 1 \\
0 & \cdots & 0 & 1 & 0
\esm,\;  \\
\matr{S}_{N} &= \bsm
0    & 0     & \cdots & 0 & 0 \\
0     & 0     & \iddots & 0 & 0     \\
\vdots  & \iddots  & \iddots & \iddots  & \vdots  \\
0 & 0 & \iddots &  & 0 \\
0 & 0     & \cdots & 0 & 1
\esm. 
\end{split}
\end{equation}
\end{example}

\subsection{Simple matrices}
Now we show that optimality conditions  hold true for simple matrices.
We take $\set{A} = \trgset{\nvar}{d}$, and consider the matrix completion problem \eqref{eq:qhankel_rankmin}, where $\matr{S}_k$ are defined in Section~\ref{sec:basis_matrices}.
\begin{lemma}\label{lem:qhankel_zero_root_projector}
Let $d' \eqdef \lfloor\frac{d}{2}\rfloor$, $\matr{P}_0 \in \bbC^{n\times n}$ be a matrix of the form 
\begin{equation}\label{eq:proj_zero}
\matr{P}_0 \eqdef \bmx  \matr{P}'_0 & \matr{0} \\ \matr{0} & \matr{0}\emx, \quad \matr{P}'_0 \in \bbC^{s\times s},
\end{equation}
where $s \le \binom{\nvar+ d'}{\nvar}$.
Then it holds that
\begin{enumeratex}
\item For any $k\in \{1,\ldots,N\}$, 
$\matr{P}^{\T}_0 \matr{S}_k \matr{P}_0  = \matr{0}$.
\item If, in addition, $s \le n'' := \binom{d''+m}{m}$, where $d'' := \lfloor\frac{d-1}{2}\rfloor$, then $\rank{\mathscr{A}(\matr{P}_0)} = N$.
\end{enumeratex}
\end{lemma}
\proofapp

\begin{example}
In the case of Hankel matrices ($\set{A} = \trgset{1}{d}$) the first bound is $s \le \lfloor \frac{d+2}{2} \rfloor$, and the first statement of Lemma~\ref{lem:qhankel_zero_root_projector} is evident from \eqref{eq:hankel_Gk}.
The second bound is $s \le \lfloor \frac{d+1}{2} \rfloor$, and a right lower block of the matrix $\matr{Q}_0\matr{H}\matr{Q}^{\top}_0$ (where $\matr{H}$ is defined in Remark~\ref{rem:equivalent_conditions_slrmc_DP15}) has the form
\[
\bmx
0     & \cdots     &0 &\Delta p_1&  \cdots & \Delta p_{r}  \\
\vdots    & \iddots     & \iddots & \iddots& \iddots &  \Delta p_{r+1}    \\
0  & \iddots  & \iddots & \iddots  & \iddots& \vdots   \\
\Delta p_1  & \iddots  & \iddots & \iddots& \iddots  &  \Delta p_{N-2}  \\
\vdots  & \iddots & \iddots & \iddots& \Delta p_{N-2}  & \Delta p_{N-1}  \\
\Delta p_{r} & \Delta p_{r+1}     & \cdots & \Delta p_{N-2} &\Delta  p_{N-1} &\Delta  p_{N}
\emx,
\]
and hence all $\Delta p_k$ are present in this block.
\end{example}

Lemma~\ref{lem:qhankel_zero_root_projector} immediately proves that the SLRMC is solved by nuclear norm minimization in very simple cases.

\begin{corollary}\label{cor:zero_root_completion}
In the completion problem~\eqref{eq:qhankel_rankmin} for the structure \eqref{eq:hankel}, let the vector $\vect{h}_d$ be of the form \eqref{eq:initial_zero_root}, with $r \le \frac{d+2}{2}$. Then we have the following.

\begin{enumeratex}
\item A solution of the rank minimization problem is given by $h_k = 0$ for $k > d$, and coincides with the solution of \eqref{eq:nnmin}.
\item If, in addition, $r < \frac{d+2}{2}$, the solution of \eqref{eq:nnmin} is unique.
\end{enumeratex}
\end{corollary}
\begin{proof}
\begin{enumeratex}
\item  Let $\matr{P}_0$ be as in \eqref{eq:proj_zero}, with $s=r$, and $\matr{P}'_0 = \matr{I}_s$. Then,
the matrix $\matr{B}_{\bbC}$ in \eqref{eq:skew_projector} has the form
$\matr{B}_{\bbC} = \matr{P}_0 \matr{A} \matr{P}^{\T}_0$.
By Lemma~\ref{lem:qhankel_zero_root_projector}, we have that $\langle \matr{S}_k,\matr{B}_{\bbC} \rangle_F = \langle \matr{P}^{\T}_0\matr{S}_k\matr{P}_0,\matr{A} \rangle_F = 0$.
The rest follows from Remark~\ref{rem:equivalent_conditions_slrmc_DP15}.
\item  Follows from Remark~\ref{rem:equivalent_conditions_slrmc_DP15}.
\end{enumeratex}
\end{proof}

\subsection{Perturbations of simple projectors}
Now we show that the conditions of Proposition~\ref{prop:opt_sufficient_complex} hold true for perturbations of simple matrices.
First, we need a simple inequality on  distance between projectors.
\begin{lemma}\label{lem:proj_dist}
Let $\matr{B} = \matr{B}^{\T}  = \matr{U}  \matr{V}^{\top}$, where $\matr{U},\matr{V}  \in \bbR^{n\times r}$ such that $\matr{U}^{\H}\matr{U} = \matr{V}^{\H}\matr{V} = \matr{I}_r$.
Let $\matr{P}_{0}$ be an orthogonal projector on a subspace of $\bbC^{n}$, such that $\rrank(\matr{P}_{0}) = r$. Then we have that 
\[
\|\matr{B} - \matr{P}_0 \matr{B} \matr{P}^{\T}_0\|^2_F \le  2\|(\matr{I}-\matr{P}_0) \matr{U}\|^2_F = \|\matr{P}-\matr{P}_0\|^2_F,
\]
where $\matr{P} = \matr{B}\matr{B}^{\H} = \matr{U}\matr{U}^{\H}$.
\end{lemma}
\begin{proof}
Since $\matr{B}\!-\!\matr{P}_0\matr{B}\matr{P}^{\T}_0 = \matr{B}\!-\!\matr{P}_0\matr{B} \!+\! \matr{P}_0\matr{B} \!-\! \matr{P}_0 \matr{B} \matr{P}_0^{\T}$,
\[
\begin{split}
&\|\matr{B} \!-\! \matr{P}_0 \matr{B} \matr{P}^{\T}_0\|^2_F  \le \|(\matr{I} \!-\! \matr{P}_0) \matr{B}\|^2_F + \|\matr{P}_0 \matr{B} (\matr{I} \!-\! \matr{P}_0^{\T})\|^2_F \\
& \le 2 \|(\matr{I}-\matr{P}_0) \matr{B}\|^2_F = 2 \|(\matr{I}-\matr{P}_0) \matr{U}\|^2_F
\end{split}
\]
Finally, 
\begin{equation}\label{eq:norm_diff_proj_int}
\begin{split}
&\!\!\|\matr{P} - \matr{P}_0\|^2_F  = \trace{\matr{P}^2 + \matr{P}_0^2 - \matr{P}\matr{P}_0 - \matr{P}_0\matr{P}}  \\
&\!\!= 2 \trace{2(\matr{I}-\matr{P}_0)\matr{P}} = 2\|(\matr{I}-\matr{P}_0)\matr{U}\|_F^2,
\end{split}\end{equation}
which completes the proof.
\end{proof}

\begin{remark}
For a symmetric matrix $\struct{\vect{p}}$, the matrix $\matr{B}_{\bbC}$ defined in Proposition~\ref{prop:opt_sufficient_complex} is also symmetric.
\end{remark}

Next, we prove a perturbation lemma that uses a bound on the distance between projectors. 
\begin{lemma}\label{lem:zero_root_perturbation}
Let  $\bfS$ and $\mathscr{A}(\matr{P})$ be defined in \eqref{eq:bfG} and \eqref{eq:fo_condition_matrix_left}.
Let $\matr{P}_0 \in \bbC^{n \times n}$ be a rank-$r$ projector matrix such that 
$\rank{\mathscr{A}(\matr{P}_0)} = N$,  $\matr{P}^{\T}_0 \matr{S}_k \matr{P}_0 = 0$ for all $k \in \{1,\ldots,N\}$, and let $\matr{Q}_0 \eqdef \matr{I}-\matr{P}_0$.

Then there exists a constant $\varepsilon$ such that for any $\matr{B}$ as in Lemma~\ref{lem:proj_dist} satisfying 
\[
\|\matr{P} - \matr{P}_0\|_F <\varepsilon, 
\]
where $\matr{P} = \matr{B}\matr{B}^{\H}$, it holds that
\begin{enumeratex}
\item The matrix $\mathscr{A}(\matr{P})$ is full row rank.
\item There exists a matrix $\matr{M} \in \bbC^{n \times n}$, $\|\matr{M}\|_2 < 1$, such that 
\eqref{eq:fo_condition_matrix} is satisfied.
\end{enumeratex}
\end{lemma}
\begin{proof}
\begin{enumeratex}
\item
Since $\mathscr{A}(\matr{P})$ depends polynomially on the entries of $\matr{P}$, and $\rrank \mathscr{A}(\matr{P}_0) = N$, the rank is preserved in a neighborhood of $\matr{P}_0$. 
Thus  $\rrank \mathscr{A}(\matr{P}) = N$ in a neighbourhood of $\matr{P}_0$.

\noindent\item
Let $\varepsilon_1>0$  be a number such that for any $\matr{P}$,  $\|\matr{P}-\matr{P}_0\|_F \le \varepsilon_1$ we have $\mathscr{A}(\matr{P}) = N$ and $0 < \delta_0 < \sigma_{min}(\mathscr{A}(\matr{P}))$ for some number $\delta_0$. (We can choose such $\varepsilon_1$ by continuity of the smallest singular value.) 
Then, in the $\varepsilon_1$-neighborhood, a solution of \eqref{eq:fo_condition_matrix} always exists.
Next, we consider the  minimum Frobenius norm solution $\matr{M}_* = \matr{M}_*(\matr{P})$ of \eqref{eq:fo_condition_matrix}, which is given by
\[ \begin{split}
&\mvec (\matr{M}_*(\matr{P}))) =  -\mathscr{A}(\matr{P})^{\dagger}  \bfS^{\T} \mvec (\matr{B}) \\ &=  -\mathscr{A}(\matr{P})^{\dagger}  \bfS^{\T} \mvec (\matr{B} - \matr{P}_0\matr{B}\matr{P}^{\T}_0), 
\end{split}
\]
where the matrix $\mathscr{A}(\matr{P})^{\dagger}$ is the pseudoinverse of $\mathscr{A}(\matr{P})$, and the last equality holds since $\matr{P}^{\T}_0 \matr{S}_k \matr{P}_0 = 0$ for all $k \in \{1,\ldots,N\}$.
Then, by Lemma~\ref{lem:proj_dist},
\begin{equation}\label{eq:norm_inequality}
\begin{split}
&\|\matr{M}_*\|_2 \le \|\mvec (\matr{M}_*(\matr{P})) \|_2 \le \\
&\|\mathscr{A}(\matr{P})^{\dagger}\|_2 \|\bfS^{\T}\!\|_2  \|\matr{B} \!-\! \matr{P}_0\matr{B}\matr{P}^{\T}_0\|_F <
\frac{\| \bfS\|_2}{\delta_0} \|\matr{P} \!-\! \matr{P}_0\|_F.
\end{split}
\end{equation}
Finally, define $\varepsilon \eqdef \min(\varepsilon_1, \frac{\delta_0}{\| \bfS\|_2})$. 
For such an $\varepsilon$, the right hand side of \eqref{eq:norm_inequality} is  less than or equal to $1$ if $\|\matr{P} - \matr{P}_0\|_F < \varepsilon$, which completes the proof.
\end{enumeratex}
\end{proof}

\section{Main results}\label{sec:results}
\subsection{Hankel case}

First, we prove a result that generalizes Theorem~\ref{thm:rank_one}.
\begin{theorem}\label{thm:rank_r}
For any $d$, $n=d+1$, and $r < \frac{d+2}{2}$ there exists  $\rho = \rho(d,r) > 0$ such that
for all $\vect{h}_d$ with $\hrank{\vect{h}_d} = r$ and  $|\lambda_k| < \rho$ in \eqref{eq:char_poly}, the solution of \eqref{eq:nnmin} is unique and coincides with \eqref{eq:canonical_completion}.
\end{theorem}
\proofapp

\begin{note}
Theorem~\ref{thm:rank_one} proves a special case of Theorem~\ref{thm:rank_r} and tells  that $\rho(d,1) = 1$ (for $\vect{h}_d \in \bbR^{d+1}$).
\end{note}

\subsection{Quasi-Hankel case}\label{sec:qhankel_results}

We first prove a lemma on the limits of projectors on column spaces of quasi-Hankel matrices of the form \eqref{eq:vandermonde_fac} (for  $\vect{z}_1,\ldots,\vect{z}_r$ in general position).

\begin{lemma}\label{lem:qhankel_limit_projectors}
Let $\set{A} = \trgset{\nvar}{d}$, $r \le \binom{\nvar + d-1}{\nvar}$, and $\{\vect{y}_1,\ldots, \vect{y}_r\} \subset \bbC^{\nvar}$ be some points.
Furthermore, assume that 
there exists a set $\set{D}$, $\trgset{\nvar}{d_0-1} \subset \set{D}\subseteq \trgset{\nvar}{d_0}$ (for some $d_0$, $0 \le d_0 \le d-1$)
such that the points $\vect{y}_1,\ldots, \vect{y}_r$ are $\set{D}$-independent.

Let $\matr{P}(\rho)$ denote the projector on the column space of $\vmatr{\set{A}}{\rho\vect{y}_1, \ldots, \rho \vect{y}_r}$.
Then we have that
\begin{itemize}
\item if $r = \binom{\nvar + d_0}{\nvar}$ (i.e., $\calD = \trgset{\nvar}{d_0}$), then
\begin{equation}\label{eq:qhankel_proj_lim_trg}
\lim_{\rho \to 0} \matr{P}(\rho) = \bmx  \matr{I}_r & \matr{0} \\ \matr{0} & \matr{0}\emx,
\end{equation}
\item if $r < \binom{\nvar + d_0}{\nvar}$ (i.e., $\calD \subset \trgset{\nvar}{d_0}$), then
\begin{equation}\label{eq:qhankel_proj_lim_general}
\lim_{\rho \to 0}\matr{P}(\rho) = \bmx  \matr{I}_K & \matr{0} & \matr{0} \\ \matr{0} & \matr{P}_2 &\matr{0}  \\ \matr{0} & \matr{0} & \matr{0}\emx,
\end{equation}
where $\matr{P}_2 \in \bbC^{L\times L}$, with $L = \binom{\nvar + d_0}{\nvar} - K$, $\rank{\matr{P}_2} = r- K$, and $K  \eqdef \binom{\nvar + d_0-1}{\nvar}.$
\end{itemize}
\end{lemma}
\proofapp

\begin{note}
Generic points $\{\vect{y}_1,\ldots, \vect{y}_{r}\}$  with $r \le \binom{\nvar + d-1}{\nvar}$ satisfy conditions of Lemma~\ref{lem:qhankel_limit_projectors}.
Indeed, the condition of $\set{D}$-independence is equivalent to 
\[
\det \vmatr{\set{D}}{\vect{y}_1, \ldots, \vect{y}_r} \neq 0,
\]
which holds for generic $\{\vect{y}_1,\ldots, \vect{y}_{r}\}$.
\end{note}

By combining Lemma~\ref{lem:zero_root_perturbation} and Lemma~\ref{lem:qhankel_limit_projectors} we obtain the following theorem.
\begin{theorem}\label{thm:qhankel_rank_r}
Let $d'' := \lfloor\frac{d-1}{2}\rfloor$ and $r \le N'' := \binom{\nvar+ d''}{\nvar}$.
Furthermore, let  $\{\vect{y}_1,\ldots, \vect{y}_{r} \}\subset \bbC^{\nvar}$ 
satisfy the conditions of Lemma~\ref{lem:qhankel_limit_projectors}.
Then there exist a constant $\rho_0 = \rho_0(\vect{y}_1,\ldots, \vect{y}_{r}) > 0$ such that for any $\rho$: $0 < \rho <  \rho_0$ and points
$\vect{z}_k$ defined as  $\vect{z}_k = \rho \vect{y}_k$, the following holds true.

For any nonzero coefficients $c_1, \ldots, c_r$ and the initial array defined as \eqref{eq:exp_array}, 
the canonical completion \eqref{eq:exp_array} (in the problem \eqref{eq:qhankel_rankmin}) is also the unique solution of the  nuclear norm minimization problem \eqref{eq:nnmin}.
\end{theorem}
\begin{proof}
By Lemma~\ref{lem:qhankel_limit_projectors} and Lemma~\ref{lem:qhankel_zero_root_projector}, the matrix $\matr{P}_0 \eqdef \lim_{\rho \to 0}\matr{P}(\rho)$ satisfies the conditions of Lemma~\ref{lem:zero_root_perturbation}.
Next, take $\varepsilon$ as in Lemma~\ref{lem:zero_root_perturbation}. Then, by Lemma~\ref{lem:qhankel_limit_projectors}, there exists $\rho_0$, such that
\[
\|\matr{P}(\rho) - \matr{P}_0\|_F< \varepsilon,
\]
which completes the proof.
\end{proof}

Note in the case $\nvar = 1$, Theorem~\ref{thm:qhankel_rank_r} is a weak version of Theorem~\ref{thm:rank_r}. However, for $\nvar > 1$, in Theorem~\ref{thm:rank_r}, it is impossible to give a uniform bound on the size of exponents (elements of $\vect{z}_k$) due to fundamental issues in multivariate polynomial interpolation.

\section{Numerical results}\label{sec:numerical}

\subsection{Hankel case}
In this case we illustrate on the examples Theorem~\ref{thm:rank_r}. Theorem~\ref{thm:rank_r} only states the existence of such radius $\rho$, and a lower bound for $\rho$ may be obtained along the lines of the proof of the theorem. However, this bound may be to small, and in this section we aim at showing on numerical experiments what would be the largest lower bound.

All the numerical experiments are reproducible and available on request.
The MATLAB package CVXOPT \cite{CVX} with default settings is used for nuclear norm minimization.

The setup for the following experiments will be similar. We take a specific $\matr{h}_d$, compute the solution of \eqref{eq:qhankel_rankmin}, and measure the  Frobenius norm between the computed solution and the canonical completion \eqref{eq:canonical_completion}.

\subsubsection{The rank-one case}
First, we consider the rank-one case. We take $n=6$, and a rank-one exponential sequence $\vect{h}_{d}$, i.e. $h_k = \lambda^{k}$, where $\lambda = a + bi$, $a,b \in (-1,1)^2$, and plot the results in Fig.~\ref{fig:cexp}.
\begin{figure}[hbt!]%
\centering
\includegraphics[height=6cm]{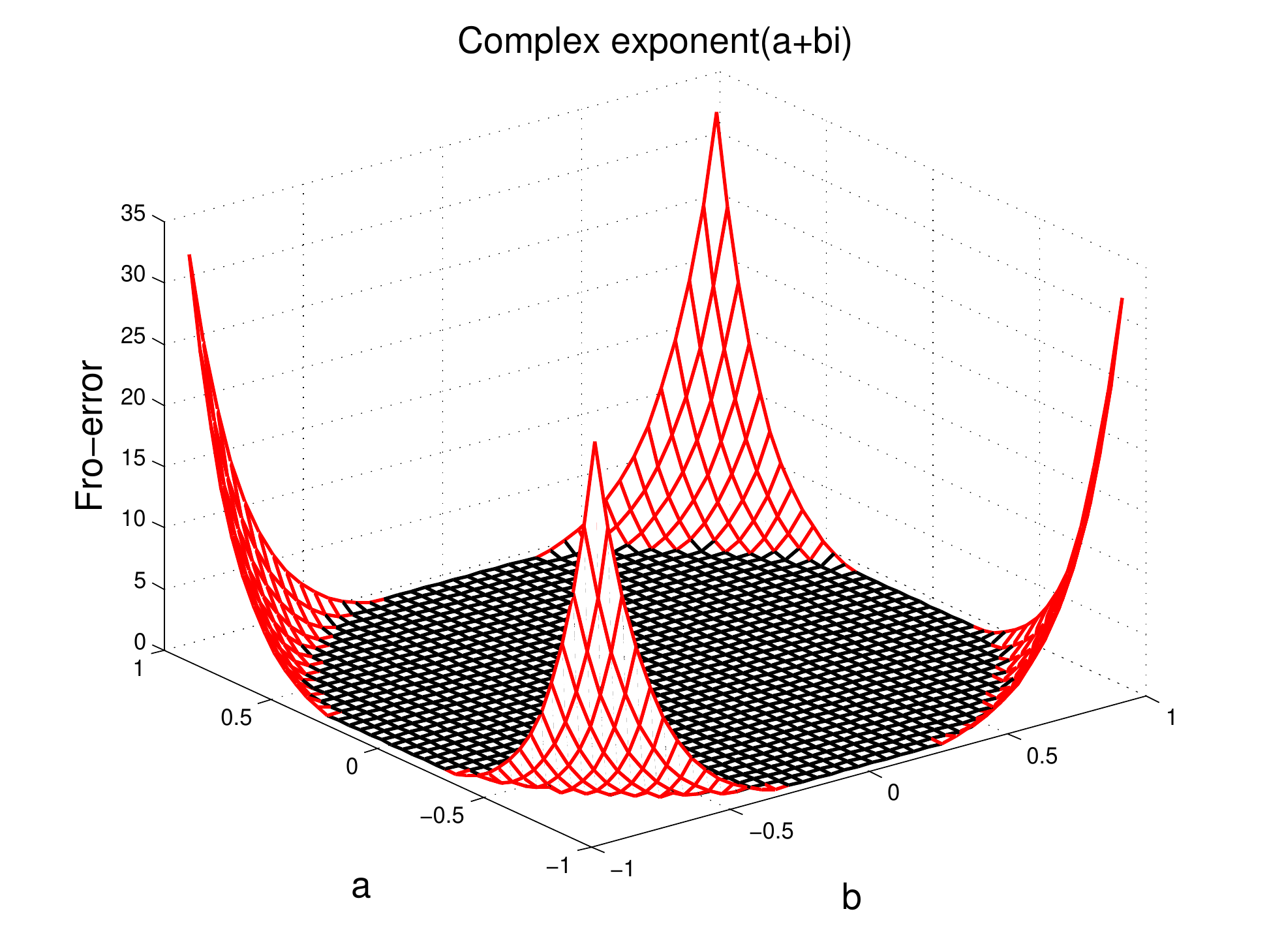}%
\caption{Nuclear norm reconstruction (Frobenius distance), $h_k = \lambda^{k}$, $\lambda = a + bi$. Black area correspond to the values less than $10^{-6}$.}
\label{fig:cexp}%
\end{figure}

In Fig.~\ref{fig:cexp}, we see that the largest lower bound for $\rho$ is very close to $1$. This coincides with the bound for real $\lambda$ given by Theorem~\ref{thm:rank_one}.

\subsubsection{The  rank-two case}
We consider real sequences $\vect{h}_d$ with $\hrank{\vect{h}_d} = 2$. In this case the following three situations are possible:
\begin{enumerate}
\item $\lambda_1, \lambda_2 \in \bbR$ (two simple real roots);
\item $\lambda_2 = \overline{\lambda_2} \not\in \bbR$  (two simple complex conjugate roots);
\item $\lambda_1 \in \bbR$, $\nu_1 = 2$ (double real root).
\end{enumerate}
The first case is considered in \cite[Fig. 1]{Dai.Pelckmans15A-nuclear}, where it is shown (numerically) that the radius $\rho$ is less than $1$. 
In this section, we examine the second and the third cases.
We generate the corresponding $\vect{h}_d$ and compute the Frobenius distance between the solutions of \eqref{eq:rankmin} and \eqref{eq:nnmin}, i.e., we compute $\|\scrS(\widetilde{\vect{p}}) - \scrS(\widehat{\vect{p}})\|_F$.

In Fig.~\ref{fig:sine}, we plot the nuclear norm reconstructions for a $6\times 6$ matrix and the last two cases.
As seen in Fig.~\ref{fig:sine}, the radius is also strictly less than $1$. 
The radius is smaller in the case of a double root and also in the case when two conjugate roots are close to each other. 
\begin{figure}[ht!]%
\centering
\includegraphics[height=6cm]{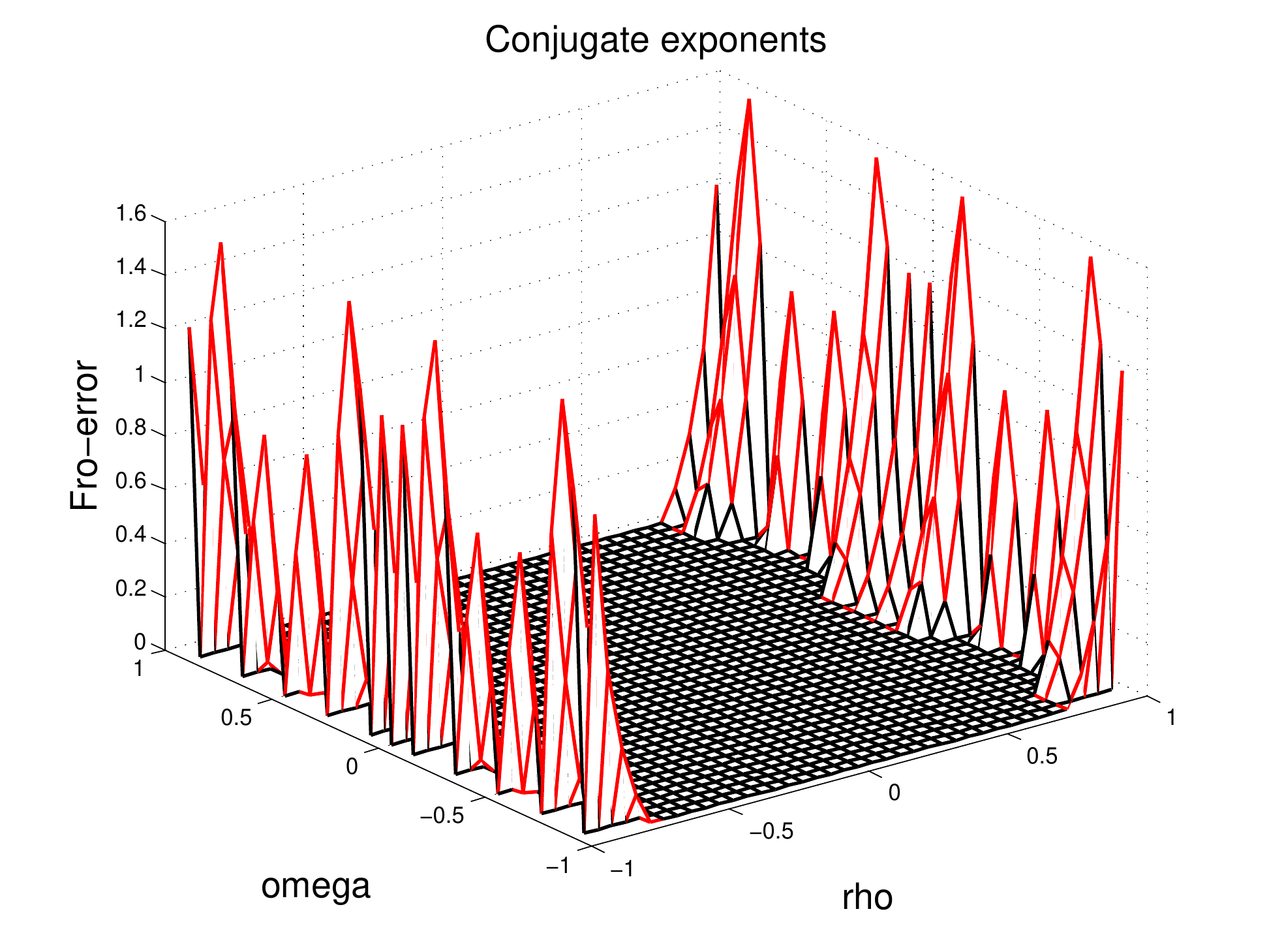}\\
\includegraphics[height=6cm]{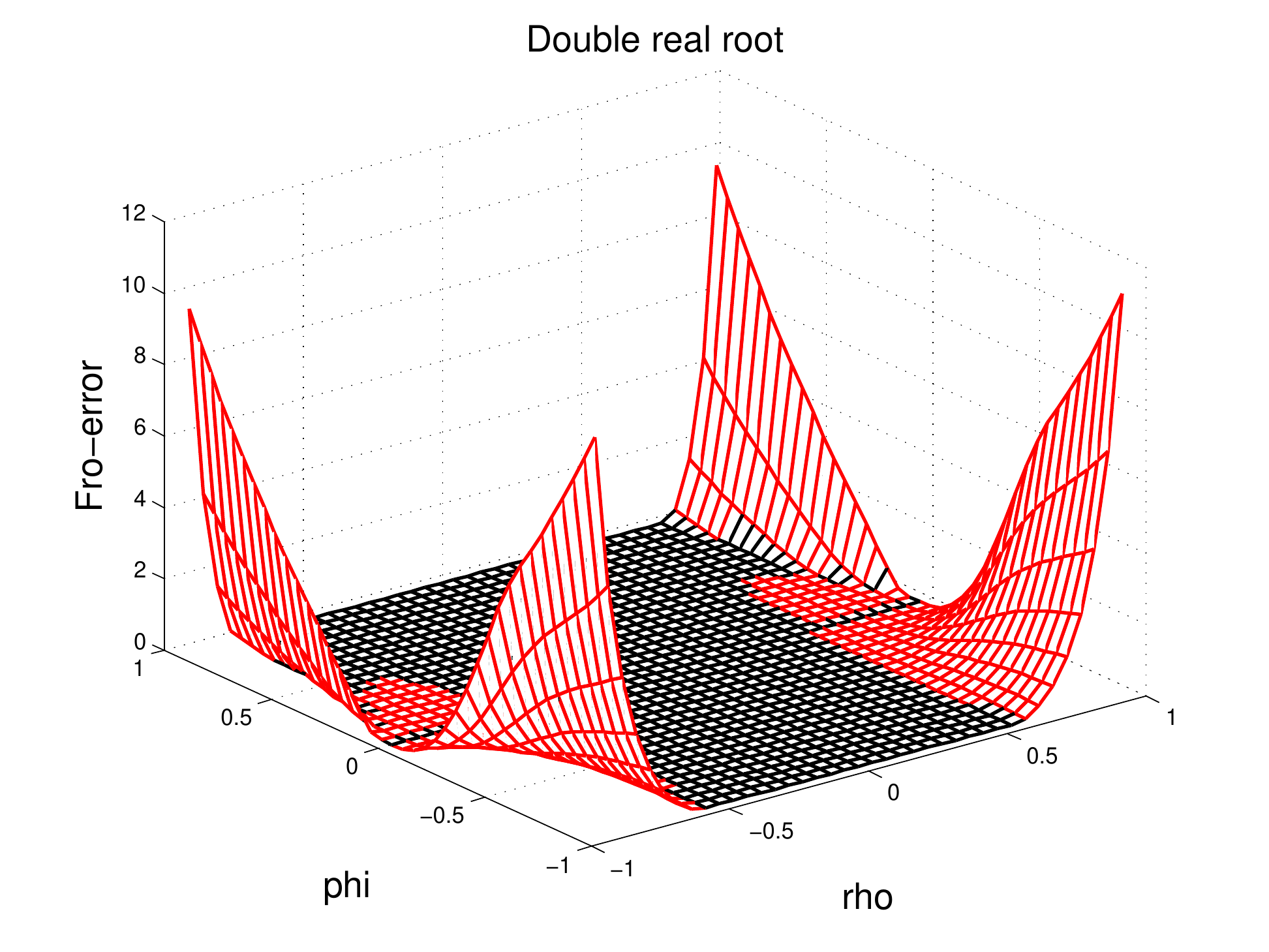}%
\caption{Nuclear norm reconstruction (Frobenius distance). Top: for two complex conjugate roots $\lambda_{1,2} = \rho \cdot\exp(\pm i\pi \omega)$ (corresponds to the damped cosine sequence $h_t = \rho^{t} \cdot \cos(\pi (\omega+1)t)$). Bottom: for a double root $\rho$ (corresponds to the damped linear function $h_t = (t \cdot \tan(0.75 \cdot \pi \cdot \varphi) +1) \cdot \rho^t$). Black areas correspond to the values less than $10^{-6}$.}%
\label{fig:sine}%
\end{figure}

\subsection{Multiple exponents}
In this experiment, we aim at estimating the radius $\rho$ based on random realizations of $\lambda_k$.
We fix $n=9$, and for each $r \in \{1,\ldots,4\}$ and for $\rho \in (0,1)$, we generate randomly  the set $\{\lambda_1,\ldots,\lambda_r\}$, such that $|\lambda_1| = \rho$ and $|\lambda_k| \le \rho$.

We consider two situations: 
\begin{itemize}
\item real roots $\lambda_k = \rho_k$ (in this case $\rho_1 = \rho$, and $\rho_k$, $k>1$ are independent and uniformly distributed in $[\rho;-\rho]$)
\item complex roots $\lambda_k = \rho_k e^{i\pi\phi_k}$, where $\rho_k$ are as in the previous example,  $\phi_k$ are independent, uniformly distributed in $[0;1]$, and independent of $\rho_k$.
\end{itemize}

We repeat the experiment $M = 100$ times, and select the maximum Frobenius error across all the realizations. The results are plotted in Fig.~\ref{fig:9roots}.

\begin{figure}[ht!]%
\centering
\includegraphics[height=6cm]{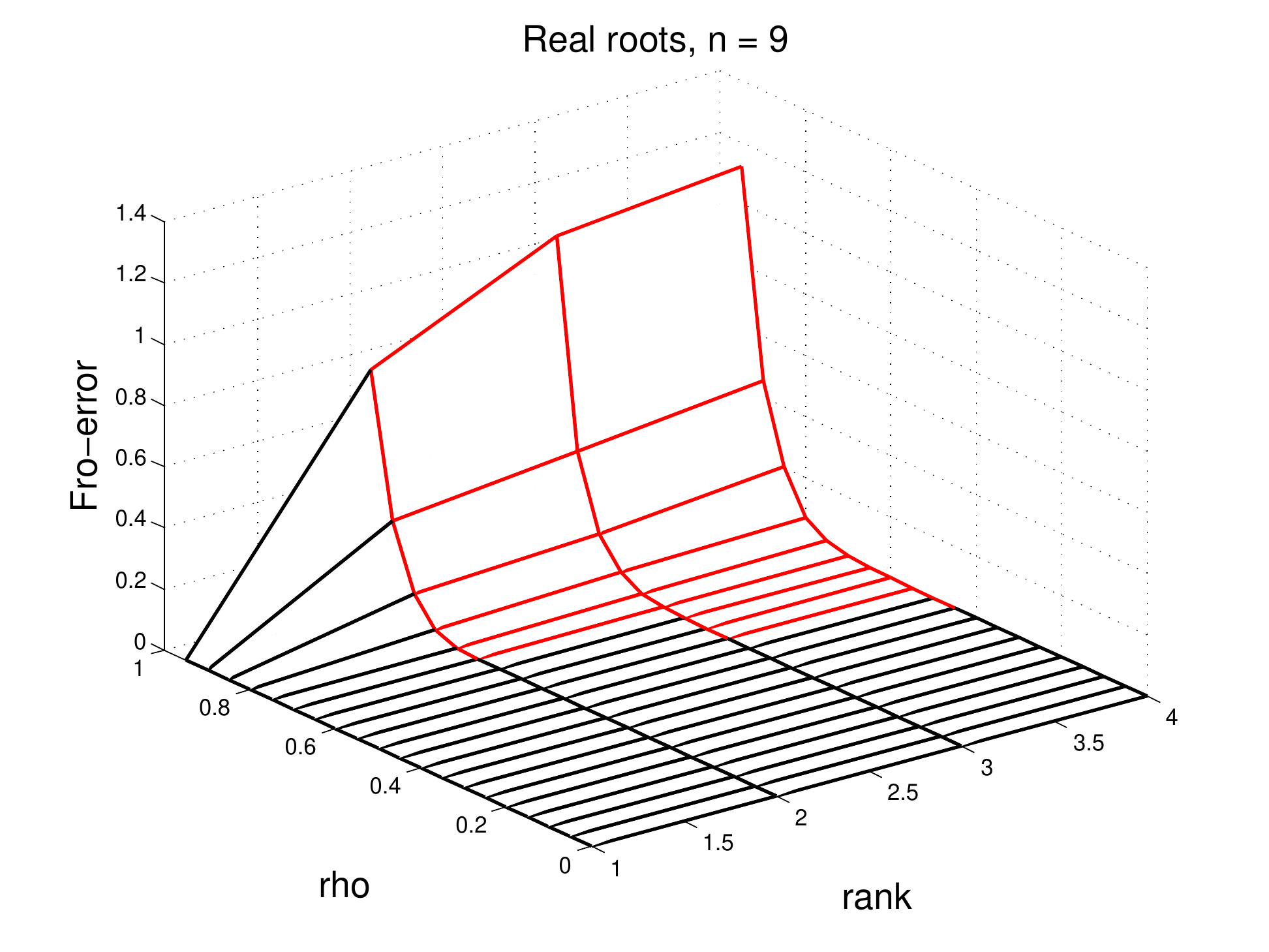}\\
\includegraphics[height=6cm]{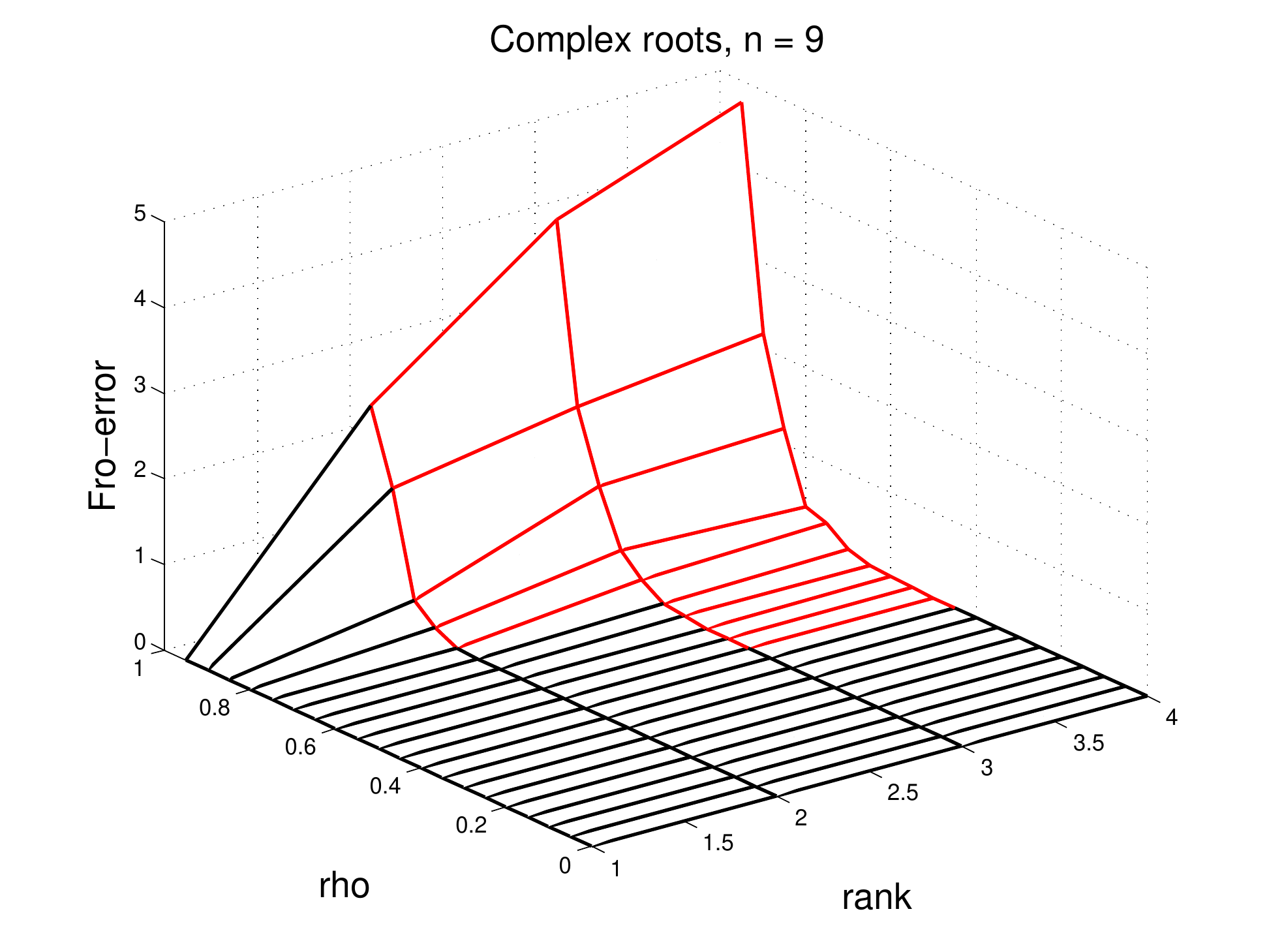}%
\caption{Nuclear norm reconstruction, $n=9$, $M=400$ (Frobenius distance). Top: real roots, bottom: complex roots.
Black area correspond to the values less than $10^{-5}$.}%
\label{fig:9roots}%
\end{figure}

The results in Fig.~\ref{fig:9roots} confirm the conclusions of Theorem~\ref{thm:rank_r}. We see that the nuclear norm heuristic works up to a certain $\rho$.
Note that the bounds are similar for the real and complex cases.

\subsection{Quasi-Hankel case}

\subsubsection{Random exponents}
In this section, we demonstrate Theorem~\ref{thm:qhankel_rank_r}. 
In the experiments, we fix $m=2$, $d=3$ and $r=3$ (thus we are dealing with the completion of the structure \eqref{eq:qhankel_trgset}).
We generate random $\vect{y}_1, \vect{y}_2,\vect{y}_3 \in \bbC^{m}$, such that 
\[
\vect{y}_k = 
\bmx
a_{k} + i b_{k}\\
e_{k} + i f_{k}\\
\emx,
\]
and $a_{k}, b_{k}, e_k, f_k$ are independent identically distributed on $[-0.5;0.5]$.

For  $\rho \in (0;1]$, we define  $\vect{z}_k$ as in Theorem~\ref{thm:qhankel_rank_r}, and consider the matrix completion with nuclear norm for the array \eqref{eq:exp_array}, where $c_k$ are chosen to be $c_1 = c_2 = c_3 = 1$.
We numerically compute the solution $\widehat{\vect{p}}$ of the nuclear norm minimization \eqref{eq:nnmin}, and 
compare it with the solution $\widetilde{\vect{p}}$ of the rank minimization \eqref{eq:rankmin} (which is known from Proposition~\ref{prop:flat_ext_exp_array}).
We repeat the experiment $M = 10$ times, and plot in Fig.~\ref{fig:qhankel_random} the dependence of $\|\widehat{\vect{p}}-\widetilde{\vect{p}}\|_2$.

\begin{figure}[ht!]%
\centering
\includegraphics[height=6cm]{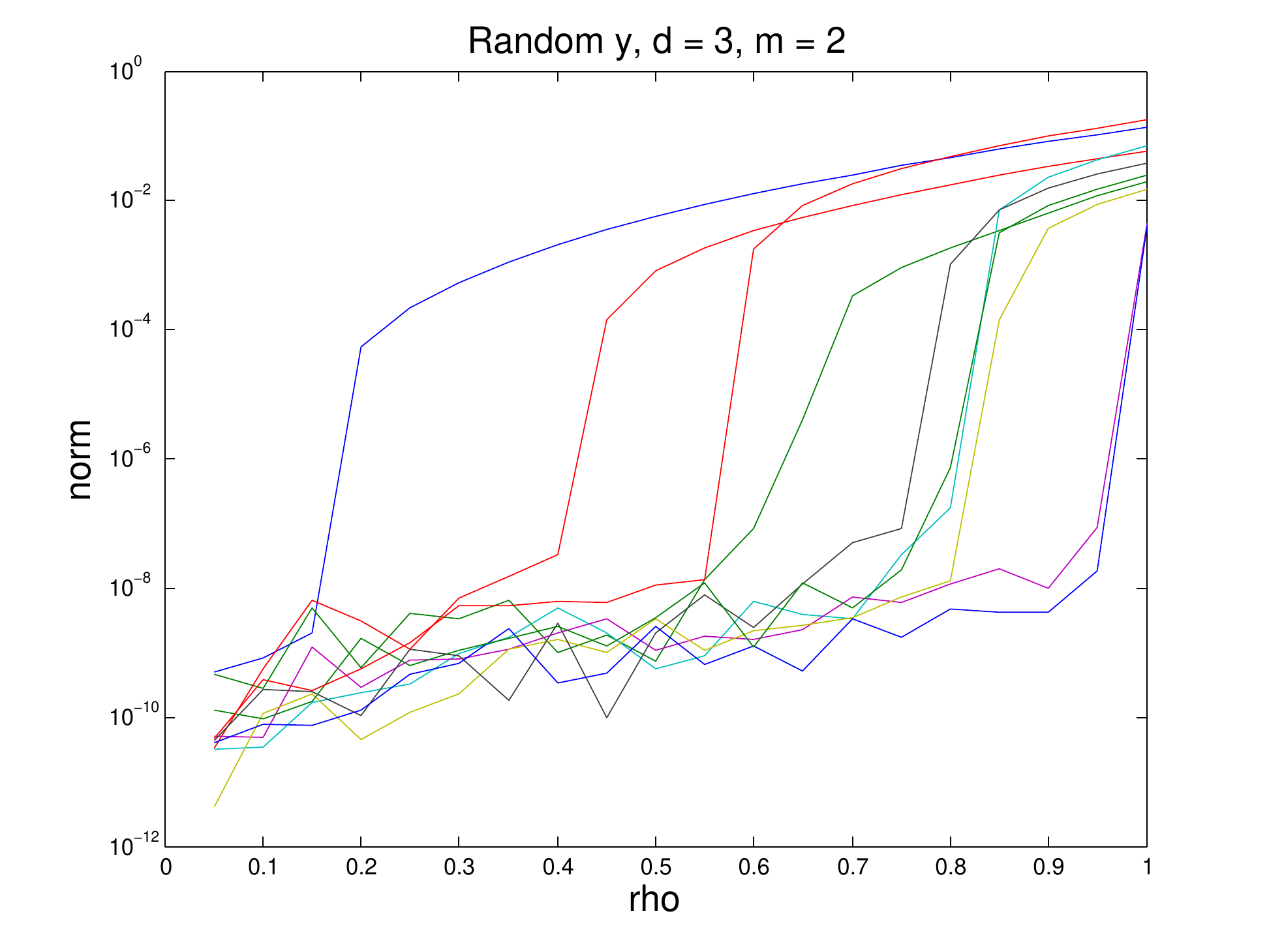}%
\caption{Nuclear norm reconstruction for quasi-Hankel matrices (parameter vector distance). Each of the curves represents a realization and shows the dependence of $\|\widehat{\vect{p}}-\widetilde{\vect{p}}\|_2$ on $\rho$.}%
\label{fig:qhankel_random}%
\end{figure}

In Fig.~\ref{fig:qhankel_random}, we see that for each realization of $\vect{y}_k$-s, there exists a radius $\rho_0$  predicted by Theorem~\ref{thm:qhankel_rank_r}.

\subsubsection{Beyond the theoretical results: nonunique decompositions}
As in the previous section, we consider the case $m=2$, $d=3$. But we choose the rank $r=4$, for which the conditions of Theorem~\ref{thm:qhankel_rank_r} no longer hold true.
For $3 \times 3 \times 3$ symmetric tensors, this is exactly the value of generic rank \cite{Iarobbino.Kanev99-Power}.

The construction of the example is based on the correspondence with symmetric tensors.
We take $d$ linearly independent vectors in $\bbR^{d}$, i.e.,
\[
\vect{v}_1,\ldots, \vect{v}_{d} \in  \bbR^{d}.
\]
Next, we consider the symmetric tensor obtained by symmetrization of a non-symmetric rank-one tensor:
\[
\mathcal{T} = \mathop{\mathrm{Sym}} (\vect{v}_1 \otimes \cdots \otimes \vect{v}_{d}).
\]
By \cite{Carlini.etal12JoA-solution}, the rank of the tensor $\mathcal{T}$ is equal to $2^{d-1}$, and the tensor admits an infinite family of decompositions, parameterized by
\[
\mathcal{T} = \frac{1}{2^{d-1} d!} \sum_{\varepsilon_2, \ldots, \varepsilon_{d} = 0 }^{1}
(-1)^{\varepsilon_2 + \cdots+ \varepsilon_d}
\otimes^{d} \vect{b}_{\varepsilon_2,\ldots,\varepsilon_d},
\]
where
\[
\vect{b}_{\varepsilon_2,\ldots,\varepsilon_d} \eqdef
(\gamma_1 \vect{v}_1 + \gamma_2 (-1)^{\varepsilon_2} \vect{v}_2 +  \cdots + (-1)^{\varepsilon_d} \gamma_d\vect{v}_{d} ),
\]
and $\gamma_1,\ldots,\gamma_d \in \bbR$ satisfy $\gamma_1\cdots\gamma_d = 1$.

In order to set up the matrix completion problem, for given $\gamma_1,\ldots,\gamma_d$, we first construct the  array $\{h_{\vect{\alpha}}\}_{\alpha \in \bbZp^{m}}$ (which is equivalent to the specific tensor decomposition), where $m \eqdef d-1$, as follows:
\begin{equation}\label{eq:array_nonunique}
h_{\vect{\alpha}} =
 \sum_{\varepsilon_2, \ldots, \varepsilon_{d} = 0 }^{1}
c_{\varepsilon_2, \ldots, \varepsilon_{d}} \vect{\lambda}^{\vect{\alpha}}_{\varepsilon_2, \ldots, \varepsilon_{d}},
\end{equation}
where
\[
\begin{split}
c_{\varepsilon_2, \ldots, \varepsilon_{d}}  & \eqdef \frac{1}{2^{d-1} d!} (-1)^{\varepsilon_2 + \cdots+ \varepsilon_d} (\vect{b}_{\varepsilon_2,\ldots,\varepsilon_d,1})^d, \\
\vect{\lambda}_{\varepsilon_2,\ldots,\varepsilon_d} & \eqdef
b_{\varepsilon_2,\ldots,\varepsilon_d,1}^{-1}
\bmx
b_{\varepsilon_2,\ldots,\varepsilon_d,2} \\
\vdots \\
b_{\varepsilon_2,\ldots,\varepsilon_d,d} \\
\emx,\\
\vect{b}_{\varepsilon_2,\ldots,\varepsilon_d} &=
\bmx
b_{\varepsilon_2,\ldots,\varepsilon_d,1} \\
\vdots \\
b_{\varepsilon_2,\ldots,\varepsilon_d,d} \\
\emx.
\end{split}
\]
Finally, we solve the nuclear norm minimization problem \eqref{eq:nnmin}, for the structure defined as in \eqref{eq:qhankel_trgset}. 

Assume for instance $(d,m)=(3,2)$ and  consider
\[
\vect{v}_1 = \bmx
4 
\\ 1
\\ 1
\emx, 
\vect{v}_2 = \bmx
1 
\\ 4
\\ 1
\emx, 
\vect{v}_3 = \bmx
1 
\\ 1
\\ 4
\emx.
\]
For this case, the nuclear norm minimization yields a $4$-rank matrix completion, which corresponds to the minimal rank. However, the obtained completion differs from the initial completion \eqref{eq:array_nonunique}, which indicates that the nuclear norm minimization resulted in a different decomposition from the infinite family of decompositions. 

We repeat the experiment for random $\vect{v}_k$, generated as
\[
\bmx
\vect{v}_1 & \vect{v}_2 & \vect{v}_3
\emx = A + E,
\]
where the elements of $E$ are i.i.d. uniform on $[-1;1]$, and $A$ is one of the two matrices:
\[
A = 3\bmx 1& 1 & 1\\1& 1 & 1\\1& 1 & 1\emx, \mbox{ or }
A = 3\bmx 1& 0 & 0\\0& 1 & 0\\0& 0 & 1\emx.
\]
By drawing $M=100$ random realizations, we check when the Euclidean norm of the vector of the last singular values $(\sigma_{5},\ldots,\sigma_{10})$ is less than $10^{-4}$ (we say that the nuclear norm heuristics succeeds in this case). We observed that in the first scenario, the nuclear norm heuristic always failed, but for the second scenario in $77$ cases the nuclear norm heuristic succeeded.

\section{Appendix}\label{sec:appendix}

\subsection{Proofs}\label{proofcomplex-sec}
\begin{proof}[Proof of Proposition~\ref{prop:flat_ext_exp_array}]
\begin{enumerate}
\item By Lemma~\ref{lem:rank_exp_array}, we have that 
\[
\rrank \hmatr{\calB}{\iarr{h}} = r.
\]
Since $\hmatr{\calB}{\iarr{h}}$ is a submatrix of $\hmatr{\calA}{\iarr{h}}$ in \eqref{eq:qhankel_rankmin}, the rank of the minimal completion is at least $r$.
Finally, by Lemma~\ref{lem:rank_exp_array}, for the canonical completion \eqref{eq:exp_array}
we have $\rrank \hmatr{\calA}{\iarr{h}} = r$.

\item[3.] In this case, the conditions of Proposition~\ref{prop:flat_ext} take place.

\item[2.] Consider a matrix of the form
\begin{equation}\label{eq:block_ext}
F = 
\bmx
\hmatr{\calB}{\iarr{h}} & \matr{C}\\
\matr{C}^{\T} & \matr{D}
\emx,
\end{equation}
where the matrix $\matr{C}$ is known with same image as $\hmatr{\calB}{\iarr{h}}$,  and the matrix $\matr{D}$ is unknown.
Since $\rrank{\hmatr{\calB}{\iarr{h}}} = r$,
for any $B$ we have that there exists unique $\matr{D}$ such that $\rrank (F) = r$ (for example, by \cite[Thm. 1.1]{Woerdeman89LAaiA-Minimal}).

Now, observe that for $\iarr{h}$ as in \eqref{eq:qhankel_rankmin}, 
\[
\hmatr{\calB^{+}}{\iarr{h}} = 
\bmx
\hmatr{\calB}{\iarr{h}} & \matr{C}\\
\matr{C}^{\T} & \hmatr{\delta \calB}{\iarr{h}}
\emx,
\]
where only $\hmatr{\delta \calB}{\iarr{h}}$ is unknown (i.e., it is of the form \eqref{eq:block_ext}). 
The canonical completion  \eqref{eq:exp_array} has rank $r$, and therefore it is the unique completion such that $\rrank \hmatr{\calB^{+}}{\iarr{h}} = r$.
Finally, since $\rrank \hmatr{\calB^{+}}{\iarr{h}} = \rrank \hmatr{\calB}{\iarr{h}}$, by Proposition~\ref{prop:flat_ext}, the overall completion is unique.
\end{enumerate}
\end{proof}

\begin{proof}[Proof of Proposition~\ref{prop:opt_sufficient_complex}]
For $\vect{p}_{\bbC} \eqdef \vect{p}_{\calR} + i \vect{p}_{\calI}$,
consider an SVD of the symmetric matrix ${\struct{\vect{p}_{\bbC}}}$, 
\begin{equation}\label{eq:S_svd_complex}
\struct{\vect{p}_{\bbC}} = \matr{U}_{\bbC} \matr{\Sigma} \matr{V}^{\H}_{\bbC}, 
\end{equation}
where $\matr{U}^{\H}_{\bbC} \matr{U}_{\bbC} = \matr{V}^{\H}_{\bbC} \matr{V}_{\bbC} =\matr{I}_r$, and $\matr{\Sigma} \in \bbR^{n\times n}$ is real diagonal.
Since ${\struct{\vect{p}_{\bbC}}}$ is symmetric, we have that $\struct{\vect{p}_{\bbC}} = \overline{\matr{V}}_{\bbC} \matr{\Sigma} \matr{U}^{\T}_{\bbC}$ is also an SVD.
 
 Next, we have that the matrices $\matr{B}_{\bbC} = \matr{B}_{\calR}+ i\matr{B}_{\calI}$ and $\matr{P}_{\bbC} = \matr{P}_{\calR}+ i\matr{P}_{\calI}$ have the form
\[
\begin{split}
 \matr{B}_{\calR}& =
\matr{U}_{\calR} \matr{V}^{\T}_{\calR}  + \matr{U}_{\calI} \matr{V}^{\T}_{\calI},\quad
\matr{B}_{\calI} = \matr{U}_{\calI} \matr{V}^{\T}_{\calR} - \matr{U}_{\calR} \matr{V}^{\T}_{\calI}, \\
 \matr{P}_{\calR}& =
\matr{U}_{\calR} \matr{U}^{\T}_{\calR}  + \matr{U}_{\calI} \matr{U}^{\T}_{\calI} = \matr{V}_{\calR} \matr{V}^{\T}_{\calR}  + \matr{V}_{\calI} \matr{V}^{\T}_{\calI} ,\\
\matr{P}_{\calI} & =
\matr{U}_{\calI} \matr{U}^{\T}_{\calR} - \matr{U}_{\calR} \matr{U}^{\T}_{\calI} = 
\matr{V}_{\calR} \matr{V}^{\T}_{\calI} -\matr{V}_{\calI} \matr{V}^{\T}_{\calR}. 
\end{split}
\]
Since the matrix $\matr{P}_{\bbC}$ is Hermitian and  $\matr{B}_{\bbC}$ is complex symmetric, we have that  $\matr{P}_{\calR}^{\T} =  \matr{P}_{\calR}$ and $\matr{P}_{\calI}^{\T} =  -\matr{P}_{\calI}$, and
$\matr{B}_{\calR}^{\T} =  \matr{B}_{\calR}$ and $\matr{B}_{\calI}^{\T} =  \matr{B}_{\calI}$.

Then from \cite{Day.Heroux01SJoSC-Solving}, we have that  the matrix $\scrS_{ext}(\vect{p}_{ext})$ admits the following SVD:
\begin{equation}\label{eq:Sext_svd}
\scrS_{ext}(\vect{p}_{ext}) = \matr{U}_{ext} \bmx\matr{\Sigma} & \matr{0} \\ \matr{0} & \matr{\Sigma}\emx \matr{V}_{ext}^{\T},
\end{equation}
where for $\matr{U}_{\bbC} =  \matr{U}_{\calR} +  i\matr{U}_{\calI}$ and $\matr{V}_{\bbC} =  \matr{V}_{\calR} +  i\matr{V}_{\calI}$
\[
 \matr{U}_{ext} \eqdef \bmx
 \matr{U}_{\calR} & - \matr{U}_{\calI} \\
 \matr{U}_{\calI} &  \matr{U}_{\calR}
\emx,  \matr{V}_{ext} \eqdef \bmx
 \matr{V}_{\calR} & - \matr{V}_{\calI} \\
 \matr{V}_{\calI} &  \matr{V}_{\calR}
\emx.
\]
We can express the matrices $\matr{B}_{ext}$ and $\matr{P}_{ext}$ (defined in \eqref{eq:skew_projector} and \eqref{eq:projector}) for the structure $\scrS_{ext}$.
Straightforward calculations show that
\[
\begin{split}
 \matr{B}_{ext}& \eqdef \matr{U}_{ext}\matr{V}_{ext}^{\T} = \bmx
\matr{B}_{\calR} & -\matr{B}_{\calI} \\
\matr{B}_{\calI} & \matr{B}_{\calR} \\
\emx \\
\matr{P}_{ext} &\eqdef \matr{U}_{ext}\matr{U}_{ext}^{\T} = \matr{V}_{ext}\matr{V}_{ext}^{\T}  = \bmx
\matr{P}_{\calR} & -\matr{P}_{\calI} \\
\matr{P}_{\calI} & \matr{P}_{\calR} \\
\emx.
\end{split}
\]
Next, from \eqref{eq:affine_ext0}, we have that
\begin{equation}\label{eq:affine_ext}
\begin{split}
&\scrS_{ext}(\vect{p}_{ext}) =
\bmx \matr{S}_{0, \calR} & -\matr{S}_{0, \calI} \\ \matr{S}_{0, \calI} & \matr{S}_{0, \calR} \emx + \\
& \sum\limits_{k=1}^{N} p_{k,\calR}
\bmx \matr{S}_{k} & 0 \\
0 & \matr{S}_{k}
\emx +
\sum\limits_{k=1}^{N} 
p_{k,\calI}
\bmx 0 & -\matr{S}_{k} \\
\matr{S}_{k} & 0
\emx,
\end{split}
\end{equation}
where $\matr{S}_0 = \matr{S}_{0, \calR} + i\matr{S}_{0, \calI}$. 

In what follows, we prove statements of the proposition.

\begin{enumerate}
\item
Analogously to \eqref{eq:subgradient_Sp_matrix} the subgradient of $g(\vect{p}_{ext}) \eqdef \|\scrS_{ext}(\vect{p}_{ext}) \|_*$  is equal to 
\[
\left\{ \bmx \vect{a}(\matr{M}_{ext}) \\ \vect{b}(\matr{M}_{ext})\emx : \matr{M}_{ext}\in \bbR^{2N \times 2N}, \|\matr{M}_{ext}\|_2 \le 1\right\},
\]
where
\[
\begin{split}
a_k(\matr{M}_{ext}) =& \left\langle \bmx \matr{S}_{k} & \matr{0} \\
\matr{0} & \matr{S}_{k}
\emx, \matr{B}_{ext} + \matr{Q}_{ext} \matr{M}_{ext} \matr{Q}_{ext2} \right\rangle_{F}, \\
b_k(\matr{M}_{ext}) =&  \left\langle\bmx \matr{0} & -\matr{S}_{k} \\
\matr{S}_{k} & \matr{0}
\emx, \matr{B}_{ext} + \matr{Q}_{ext} \matr{M}_{ext} \matr{Q}_{ext2} \right\rangle_{F},
\end{split}
\]
and for $\matr{Q}_{\bbC} =\matr{Q}_{\calR} + i\matr{Q}_{\calI}  \eqdef \matr{I}_n - \matr{P}_{\bbC}$
\[
\matr{Q}_{ext} = \bmx
\matr{Q}_{\calR} & -\matr{Q}_{\calI} \\
\matr{Q}_{\calI} & \matr{Q}_{\calR} \\
\emx, \matr{Q}_{ext2} = \bmx
\matr{Q}_{\calR} & \matr{Q}_{\calI} \\
-\matr{Q}_{\calI} & \matr{Q}_{\calR} \\
\emx
\]
are the projectors on the left and right nullspace of  $\scrS_{ext}(\vect{p}_{ext})$.
Immediately, we have that
\[
\begin{split}
2\langle \matr{S}_{k}, \matr{P}_{\bbC} \rangle_F 
&= 
\left\langle \bmx \matr{S}_{k} & \matr{0} \\
\matr{0} & \matr{S}_{k}
\emx, \matr{B}_{ext} \right\rangle_{F} \\
&+ i
\left\langle
\bmx \matr{0} & -\matr{S}_{k} \\
\matr{S}_{k} & \matr{0}
\emx, \matr{B}_{ext}\right\rangle_F.
\end{split}
\]
Now let us define the matrices 
\begin{equation}\label{eq:M_ext}
\begin{split}
\matr{M}_{ext} &= 
\bmx
\matr{M}_{1} & -\matr{M}_{3} \\
\matr{M}_{2} & \matr{M}_{4}
\emx,\;
\matr{M}_{\calR} \eqdef \frac{\matr{M}_1+\matr{M}_4}{2},\\
 \matr{M}_{\calI} &\eqdef \frac{\matr{M}_2+\matr{M}_3}{2},\;
\matr{M}_{\bbC} \eqdef \matr{M}_{\calR} + i \matr{M}_{\calI}. 
\end{split}
\end{equation}
Easy calculations show that the matrix $\matr{Q}_{ext} \matr{M}_{ext} \matr{Q}_{ext2}$ can be expressed as
\begin{equation}\label{eq:M_ext_expr}
\matr{Q}_{ext} \matr{M}_{ext} \matr{Q}_{ext2} = \bmx \matr{C}_1 & -\matr{C}_3\\ \matr{C}_2 & \matr{C}_4\emx,
\end{equation}
where 
\[
\begin{smallmatrix}
\matr{C}_1  =  \matr{Q}_{\calR}  \matr{M}_{1} \matr{Q}_{\calR}   + \matr{Q}_{\calR} \matr{M}_{3} \matr{Q}_{\calI}- \matr{Q}_{\calI} \matr{M}_{2} \matr{Q}_{\calR}   + \matr{Q}_{\calI} \matr{M}_{4} \matr{Q}_{\calI},\\
\matr{C}_2  =  \matr{Q}_{\calR} \matr{M}_{2} \matr{Q}_{\calR} - \matr{Q}_{\calR} \matr{M}_{4} \matr{Q}_{\calI}+ \matr{Q}_{\calI}  \matr{M}_{1} \matr{Q}_{\calR}   + \matr{Q}_{\calI} \matr{M}_{3} \matr{Q}_{\calI}  ,\\
\matr{C}_3  =  \matr{Q}_{\calR} \matr{M}_{3} \matr{Q}_{\calR}-\matr{Q}_{\calR}  \matr{M}_{1} \matr{Q}_{\calI} + \matr{Q}_{\calI} \matr{M}_{4} \matr{Q}_{\calR} + \matr{Q}_{\calI} \matr{M}_{2} \matr{Q}_{\calI},\\
\matr{C}_4  =  \matr{Q}_{\calR} \matr{M}_{4} \matr{Q}_{\calR}+ \matr{Q}_{\calR} \matr{M}_{2} \matr{Q}_{\calI}  - \matr{Q}_{\calI} \matr{M}_{3} \matr{Q}_{\calR} +\matr{Q}_{\calI}  \matr{M}_{1} \matr{Q}_{\calI},\\
\end{smallmatrix}
\]
and 
\[
\matr{Q}_{\bbC} \matr{M}_{\bbC} \matr{Q}^{\top}_{\bbC} = \frac{\matr{C}_1+\matr{C}_4}{2}  +i\frac{\matr{C}_2+\matr{C}_3}{2}.
\]
Then, immediately,
\[
\begin{split}
&2\langle \matr{S}_k, \matr{B}_{\bbC}+\matr{Q}_{\bbC} \matr{M}_{\bbC} \matr{Q}_{\bbC}\rangle_F\\
= &
\left\langle \bmx\matr{S}_k& \matr{0} \\
\matr{0} & \matr{S}_{k}
\emx, \matr{B}_{ext} - \matr{Q}_{ext} \matr{M}_{ext} \matr{Q}_{ext} \right\rangle_{F} \\
+& i
\left\langle
\bmx \matr{0} & -\matr{S}_{k} \\
\matr{S}_{k} & \matr{0}
\emx, \matr{B}_{ext} - \matr{Q}_{ext} \matr{M}_{ext} \matr{Q}_{ext} \right\rangle_F.
\end{split}
\]
Hence, we obtain that
\[
a_k(\matr{M}_{ext})   + i b_k(\matr{M}_{ext}) =
 \left\langle  \matr{S}_{k} 
, \matr{B}_{\bbC} + \matr{Q}_{\bbC} \matr{M}_{\bbC} \matr{Q}_{\bbC}^{\top} \right\rangle_{F} 
\]
Finally, it is easy to see that for $\matr{M}_{\bbC}$ defined as in \eqref{eq:M_ext}
\begin{equation}\label{eq:M_norms}
\begin{split}
&\|\matr{M}_{\bbC}\|_2 = \left\| \bmx \matr{M}_{\calR} & -\matr{M}_{\calI} \\ \matr{M}_{\calI} & \matr{M}_{\calR}  \emx \right\|_2  \\
&= \left\| \frac{\matr{M}_{ext} + \bsm \matr{0} & -\matr{I} \\ \matr{I} & \matr{0} \esm \matr{M}_{ext} \bsm \matr{0} & -\matr{I} \\ \matr{I} & \matr{0} \esm}{2}  \right\|_2 \\&\le \|\matr{M}_{ext}\|_2,
\end{split}
\end{equation}
where the equality takes place if $\matr{M}_1 = \matr{M}_4$ and $\matr{M}_2 = \matr{M}_3$.
Hence, if  $\|\matr{M}_{ext}\|_2 < 1$ and then  $\|\matr{M}_{\bbC}\|_2 < 1$, according to \eqref{eq:M_norms}.
Vice versa, if we are given $\matr{M}_{\bbC} = \matr{M}_{\calR} + i \matr{M}_{\calI}$ with $\|\matr{M}_{\bbC}\|_2 < 1$, by taking $\matr{M}_1 = \matr{M}_4 = \matr{M}_{\calR}$ and $\matr{M}_2 = \matr{M}_3 = \matr{M}_{\calI}$, we obtain $\matr{M}_{ext}$ with $\|\matr{M}_{ext}\|_2  = \|\matr{M}_{\bbC}\|_2$, which completes the proof.

\item
Let $\matr{M}_{ext}$ be as in \eqref{eq:M_ext},  with $\matr{M}_1 = \matr{M}_4 = \matr{M}_{\calR}$, $\matr{M}_2 = \matr{M}_3 = \matr{M}_{\calI}$. Then, in view of \eqref{eq:M_ext_expr}, we have that
\[
\matr{Q}_{\bbC}  \matr{M}_{\bbC} \matr{Q}^{\T}_{\bbC} = \matr{0} \iff
\matr{Q}_{ext} \matr{M}_{ext} \matr{Q}_{ext2} = \matr{0},
\]
which completes the proof.
\end{enumerate}
\end{proof}

\begin{proof}[Proof of Lemma~\ref{lem:qhankel_zero_root_projector}]
\begin{enumeratex}
\item
Denote $\set{B} = \trgset{\nvar}{d'}$ (in this case, $s \le \# \set{B}$).
Since $2\set{B} \subseteq \set{A}$, from \eqref{eq:qhankel_Gk}, we have that for any $k \in \{1,\ldots, N\}$
\begin{equation}\label{eq:qhankel_Gk_form}
\matr{S}_k = 
\left[
\begin{array}{c|c}
\matr{0}_{s\times s} & * \\\hline
* & * \\
\end{array}
\right].
\end{equation}
Therefore, $\matr{P}^{\T}_0 \matr{S}_k \matr{P}_0 = \matr{0}$.

\noindent\item
First, denote $\set{C} = \trgset{\nvar}{d''}$ (in this case, $n'' = \# \set{C}$). Then the matrix $\matr{Q}_0$  has the form
\[
\matr{Q}_0 = \bmx  \matr{I}_{s} - \matr{P}'_0 & \matr{0} \\ \matr{0} & \matr{I}_{n-s}\emx,
\]
Now, as in Remark~\ref{rem:equivalent_conditions_slrmc_DP15}, we take an arbitrary $\matr{H}= \sum\limits_{k=1}^{N} \Delta p_k \matr{S}_k \neq \matr{0}$. It is sufficient to prove that the following block of the matrix $\matr{Q}_0\matr{H}\matr{Q}^{\T}_0$ is nonzero:
\[
\matr{F}\matr{H}\matr{F}^{\top}\neq \matr{0}, \quad \mbox{where}\quad
\matr{F} \eqdef \bmx   \matr{0} & \matr{I}_{n-n''}\emx.
\]
It is easy to see that the matrix $\matr{F}\matr{H}\matr{F}^{\top}$ has the form 
\[ 
\matr{F} \matr{H} \matr{F}^{\top} = \hmatr{\set{D}}{\iarr{f}},
\]
where the values of the array $\arr{f} \eqdef \{f_{\vect{\alpha}}\}_{\alpha \in 2\set{A}}$ are 
\[
{f}_{\vect{\alpha}} \eqdef 
\begin{cases}
0, & {\vect{\alpha}} \in \set{A},  \\
\Delta p_k, & \vect{\alpha} = \vect{\beta}_{k},
\end{cases}
\]
$\vect{\beta}_{k}$ are defined as in \eqref{eq:ordering_missing_values}, and $\set{D} \eqdef \set{A} \setminus \set{C}$. Finally, we have that
\[
2\set{D} = \trgset{\nvar}{2d} \setminus \trgset{\nvar}{2d'' + 1} \supseteq2\set{A} \setminus \set{A},
\]
and therefore the matrix $\matr{F} \matr{H} \matr{F}^{\T}$ contains all $\Delta p_k$  and  
 is nonzero for any considered $\matr{H}$.
\end{enumeratex}
\end{proof}

\begin{proof}[Proof of Theorem~\ref{thm:rank_r}]
By Lemma~\ref{lem:zero_root_perturbation}, we only need to show that for any $\varepsilon > 0$  there exists $\rho$ such that for all $\lambda_k$ satisfying the statements of the theorem it holds that 
\[
\|\matr{P} - \matr{P}_0\|_F <\varepsilon, 
\]
where 
\[
\matr{P}_0 = \bmx \matr{I}_r & \matr{0} \\ \matr{0} & \matr{0}\emx \in \bbR^{n \times n}.
\]
The distance between projectors can be expressed as
\[
\|\matr{P} - \matr{P}_0\|_F = \|(\matr{I}-\matr{P}) - (\matr{I}-\matr{P}_0)\|_F =
2 \|\matr{P}_0 \matr{U}_{\bot}\|_F,
\]
where  $\matr{U}_{\bot}$ is an orthonormal basis on the nullspace of $\matr{P}$ (the last equality follows from \eqref{eq:norm_diff_proj_int}).

The nullspace of $\matr{P}$ coincides with the image of  the matrix
\[
\matr{K} = 
\bmx
        q_0     & 0      & \ldots &0        \\
        \vdots  & q_0    & \ddots &\vdots   \\
        q_{r}   & \vdots & \ddots &0        \\
        0       & q_{r}  & \ddots &q_0      \\
        \vdots  &\ddots  & \ddots &\vdots   \\
        0       & \ldots & 0      &q_{r}   
\emx \in \bbC^{n\times (n-r)},
\]
where $q(z)$ is the characteristic polynomial \eqref{eq:char_poly} (with $q_r = 1$).
Then $\matr{U}_{\bot} \in \bbC^{n\times(n-r)}$ can be found from an SVD $\matr{K} = \matr{U}_{\bot} \matr{\Sigma} \matr{V}^{\H}_{\bot}$, and, by submultiplicativity of matrix norms, we have that 
\[
\frac{\|\matr{P} - \matr{P}_0\|^2_F}{2} \le  \|\matr{P}_0 \matr{K}\|_F^2 \|\matr{\Sigma}^{-1}\|^2_F 
\]
The entries of the matrix $\matr{P}_0 \matr{K}$ are $q_j$, for $j = \{0,\ldots, r-1\}$.
By Vieta's formulae,  these $q_j$ are homogeneous polynomials in the roots $\lambda_k$.
Therefore, there exists a universal constant $C$ such that for all $\lambda_k$, satisfying $|\lambda_{k}| < 1$ it holds that $|q_j| < C \max (|\lambda_{k}|)$.
Hence,
\[
\frac{\|\matr{P} - \matr{P}_0\|^2_F}{2} \le \frac{C(n-r) }{\min{\sigma^2_{min}(\matr{K})}}\max (|\lambda_k|).
\]
From the well-known bounds on eigenvalues of Toeplitz matrices \cite{Grenander.Szego84-Toeplitz}, we know that
\[
\min{\sigma^2_{min}(\matr{K})} \ge \min_{|z| = 1, z \in \bbC} |q(z)|^2,
\]
which can be bounded from below by a constant if the roots $\lambda_k$  satisfy $|\lambda_k| < \rho_0$
for  $0 < \rho_0 < 1$. Hence, there exists a constant $\widetilde{C}$ such that for all $|\lambda_k| < \rho_0$ the following inequality holds true
\[
{\|\matr{P} - \matr{P}_0\|^2_F} \le \widetilde{C} \max (|\lambda_k|).
\]
\end{proof}

\begin{proof}[Proof of Lemma~\ref{lem:qhankel_limit_projectors}]
Denote $M = \# \set{A}$ and assume that
$\set{D} = \{ \vect{\alpha}_{1}, \ldots, \vect{\alpha}_{r} \}$ and $\delta \set{D} = \{ \vect{\beta}_{1}, \ldots, \vect{\beta}_J \}$.
Since $\vect{y}_1,\ldots, \vect{y}_r$ are $\set{D}$-independent, there exist coefficients $[w_{i,j}]_{i,j=1}^{r,J}$  such that the polynomials
\[
p_j(\vect{y}) \eqdef \vect{y}^{\vect{\beta}_j} - \sum\limits_{i=1}^{r} w_{i,j} \vect{y}^{\vect{\alpha}_i}
\]
vanish on $\{\vect{y}_1,\ldots, \vect{y}_r\}$. 
The polynomials $p_j(\vect{y})$  constitute a so-called \emph{border basis} \cite[Ch. 2]{Stetter04-Numerical} of the interpolating ideal of $\{\vect{y}_1,\ldots, \vect{y}_r\}$ (the ideal of polynomials vanishing at  $\{\vect{y}_1,\ldots, \vect{y}_r\}$).

Now, define the following polynomials:
\[
p_{j,\rho}(\vect{y}) \eqdef  \vect{y}^{\vect{\beta}_j} - \sum\limits_{i=1}^{r} w_{i,j} \vect{y}^{\vect{\alpha}_i} \rho^{|\vect{\beta}_j| - |\vect{\alpha}_i|},
\]
which for $\rho >0$ are equal to $\rho^{|\vect{\beta}_j|} p_j(\rho^{-1}\vect{y})$ and therefore vanish  $\{\rho\vect{y}_1,\ldots, \rho\vect{y}_r\}$. Moreover, $p_{j,\rho}(\vect{y}) $ constitute a border basis of the interpolating ideal of $\{\rho\vect{y}_1,\ldots, \rho\vect{y}_r\}$.

Now let us build the left kernel $\calK_{\rho} \eqdef \lker{\scrV_{\set{A}} (\rho\vect{y}_1, \ldots, \rho \vect{y}_r)}$ of the corresponding Vandermonde matrix.
In what follows, we construct a basis of $\calK_{\rho}$.
Consider  sets of multiindices
\[
\begin{split}
\set{E}_1  & =  \{\vect{\epsilon}_{1,1}, \ldots, \vect{\epsilon}_{1,n_1}\}, \cdots
,\set{E}_J   =  \{\vect{\epsilon}_{J,1}, \ldots, \vect{\epsilon}_{J,n_J}\}, \\
\end{split}
\]
such that $n_1 + \ldots + n_J = M-r$ and 
\[
(\set{E}_1 + \{\vect{\beta}_1\}) \cup \cdots \cup (\set{E}_M + \{\vect{\beta}_M\}) = \set{A} \setminus \set{D}.
\]
Then we have that the polynomials $\{q_{j,i,\rho}\}_{j=1,i=1}^{J,n_{j}}$ defined as
\[
q_{j,i,\rho}(\vect{y}) \eqdef \vect{y}^{\vect{\epsilon}_{j,i}} p_{j, \rho} (\vect{y})
\]
are linearly independent.
Finally, the monomials of each polynomial $q_{j,i,\rho}(\vect{y})$ have degree at most $d$, and therefore,
there exist vectors $\vect{q}_{j,i,\rho} \in \bbC^{M}$ such that
\[
q_{j,i,\rho}(\vect{y}) = \bmx
\vect{y}^{\vect{\alpha}_1} & \cdots & \vect{y}^{\vect{\alpha}_M} 
\emx \vect{q}_{j,i,\rho}.
\]
Since all $q_{i,j,\rho}(\vect{y})$ vanish on $\vect{y}_1,\ldots, \vect{y}_r$, and $q_{i,j}$ are linearly independent, we have that
\[
\calK_{\rho} = \colspan \{\vect{q}_{j,i,\rho}\}_{ j=1, i=1}^{J,n_J}.
\] 
Now we  calculate the limit of $\calK_{\rho}$ as $\rho \to 1$.
Consider the cases from the statement of the lemma.

\begin{enumeratex}
\item If $r = \binom{\nvar + d_0}{\nvar}$,  we have that $\set{D} = \trgset{\nvar}{d_0}$, and $\delta \set{D} = \degset{\nvar}{d_0+1}$. Hence, $|\vect{\beta}| > |\vect{\alpha}|$ for all $\vect{\beta} \in \delta\set{D}$, $\vect{\alpha} \in \set{D}$, and we have that 
\[
\lim_{\rho \to 0} q_{j,i,\rho}(\vect{y}) =  q_{j,i,0}(\vect{y}) = \vect{y}^{\varepsilon_{j,i} + \vect{\beta}_j}.
\]
Since all sums $\varepsilon_{j,i} + \beta_j$ are distinct, we have that
\[
\lim_{\rho \to 0} \calK_{\rho} = \colspan \bmx 0 \\ I_{N-r}\emx,
\]
which implies \eqref{eq:qhankel_proj_lim_trg}.

\item Let $r = \binom{\nvar + d_0}{\nvar}$. Without loss of generality, we assume that elements of $\set{D}$ and $\delta \set{D}$ are ordered such that 
\[
\begin{split}
&\trgset{\nvar}{d_0-1} =  \{ \vect{\alpha}_{1}, \ldots, \vect{\alpha}_{K} \},\quad  
\set{C}_1  =  \{ \vect{\alpha}_{K+1}, \ldots, \vect{\alpha}_{r} \} \subset \degset{\nvar}{d_0},  \\
&\set{C}_2 =  \{ \vect{\beta}_{1}, \ldots, \vect{\beta}_{T} \} \subset \degset{\nvar}{d_0},\quad 
\set{C}_3 =  \{ \vect{\beta}_{T+1}, \ldots, \vect{\beta}_{J} \} \subset \degset{\nvar}{d_0+1}, \\
\end{split}
\]
and $T = L+K-r$, $\set{D} = \trgset{\nvar}{d_0-1} \cup \set{C}_1$, $\delta \set{D} = \set{C}_2\cup \set{C}_3$.
Note that since $\set{C}_3 = \delta \set{D} \setminus \degset{\nvar}{d_0}$, we have that $\set{C}_3 = \delta 
\set{C}_1$. Now we have that
\[
\lim_{\rho \to 0} q_{j,i,\rho}(\vect{y}) = q_{j,i,0}(\vect{y}) =
\begin{cases}
\vect{y}^{\varepsilon_{j,i}} (\vect{y}^{\vect{\beta}_j} - \sum\limits_{l=K+1}^{r} w_{l,j} \vect{y}^{\vect{\alpha}_l}), & j \le T,\\
\vect{y}^{\varepsilon_{j,i} + \vect{\beta}_j}, & j > T.
\end{cases}
\]
Finally, note that we can order $\vect{\epsilon}_{j,i}$ in such a way that $\vect{\epsilon}_{j,1} = \vect{0}$ for all $j$.
Taking into account that $\delta \set{C}_1 = \set{C}_3$, we can cancel lower monomials in $q_{j,i,0}(\vect{y})$, $i > 1$ without changing the linear space spanned by $q_{i,j,0}$.
More precisely, the set of polynomials $\{\widetilde{q}_{i,j}\}^{J,n_J}_{j=1,i=1}$, defined as
\[
\widetilde{q}_{j,i}(\vect{y}) =
\begin{cases}
q_{j,i,0}(\vect{y}), & (i = 1) \mbox{ and } (j \le T),\\
\vect{y}^{\varepsilon_{j,i} + \vect{\beta}_j}, & (i > 1) \mbox{ or } (j > T),\\
\end{cases}
\]
spans the same subspace as $\{\widetilde{q}_{j,i}\}^{J,n_J}_{j=1,i=1}$.
Therefore,  the limit of the kernel has the form
\[
\lim_{\rho \to 0} \calK_{\rho} = 
\colspan 
\bmx 
*    & \matr{0} \\ 
\matr{I}_T  & \matr{0} \\  
     & \matr{I}_{N-(K+L)}\emx,
\]
where $*$ denotes the elements of $q_{i,j,0}(\vect{y})$ for $j=1$ (\textit{i.e.}, the first $T$ columns correspond to $\widetilde{q}_{j,i}(\vect{y})$ for $i = 1$ and $j \le T$).
Hence, the projector on the orthogonal complement of $\lim_{\rho \to 0} \calK_{\rho}$ has the form \eqref{eq:qhankel_proj_lim_general}.
\end{enumeratex}
\end{proof}

\section*{Acknowledgment}
This work is supported by the ERC project ``DECODA'' no.320594, in the frame of the European program FP7/2007-2013. We thank L. Dai and K. Pelckmans for motivating discussions.

\bibliographystyle{alpha}
\bibliography{matcompl,Gretsi2015/completion}

\end{document}